\pdfoutput=1
\documentclass[a4paper,11pt]{amsart}
\usepackage[hmarginratio={1:1},vmarginratio={1:1},lmargin=80.0pt,tmargin=80.0pt]{geometry}

\usepackage{calligra}
\usepackage{fontenc}
\usepackage[ugly]{nicefrac}

\usepackage{mathtools}
\mathtoolsset{mathic}
\usepackage[final]{microtype}
\usepackage{multicol}
\usepackage{tabularx,tabu}
\usepackage{latexsym,exscale,enumitem,amsfonts,amssymb,mathtools}
\usepackage{amsmath,amsthm,amsfonts,amssymb,amscd,textcomp,bbm}
\usepackage{stmaryrd}
\SetSymbolFont{stmry}{bold}{U}{stmry}{m}{n}
\usepackage[normalem]{ulem}
\usepackage{thmtools}
\usepackage{etex}
\usepackage{fancybox}
\usepackage[table]{xcolor}
\usepackage[norelsize]{algorithm2e}
\usepackage{mathrsfs}
\DeclareMathAlphabet{\mathscrbf}{OMS}{mdugm}{b}{n}

\usepackage{caption} 
\usepackage[all]{xy}
\SelectTips{cm}{}

\definecolor{myorange}{RGB}{225,127,0}
\definecolor{mygreen}{RGB}{0,225,0}
\definecolor{mypurple}{RGB}{128,0,128}
\definecolor{myred}{RGB}{255,0,0}
\definecolor{myblue}{RGB}{0,0,195}
\definecolor{myyellow}{RGB}{210,210,0}
\definecolor{mycream}{RGB}{200,200,150}

\definecolor{dummy}{RGB}{100,100,100}
\definecolor{mygray}{gray}{0.7}

\definecolor{orchid}{RGB}{143,40,194}
\definecolor{lava}{RGB}{207,16,32}

\definecolor{mydarkblue}{RGB}{10,10,170}

\definecolor{sea}{RGB}{46,139,87}
\definecolor{tomato}{RGB}{255,99,71}

\usepackage{tikz}
\tikzset{anchorbase/.style={baseline={([yshift=-0.5ex]current bounding box.center)}},
  tinynodes/.style={font=\tiny,text height=0.75ex,text depth=0.15ex},
  bstrand/.style={line width=1.5, color=myblue},
  rstrand/.style={line width=1.5, color=myred},
  ystrand/.style={line width=1.5, color=myyellow},
  dstrand/.style={line width=1.5, color=black},
  seastrand/.style={line width=1.5, color=sea},
  tomstrand/.style={line width=1.5, color=tomato},
}
\usetikzlibrary{cd}
\usetikzlibrary{decorations}
\usetikzlibrary{decorations.markings}
\usetikzlibrary{decorations.pathreplacing}
\usetikzlibrary{decorations.pathmorphing}
\usetikzlibrary{shapes,positioning,matrix,calc}
\usetikzlibrary{shapes.callouts}
\usetikzlibrary{patterns}
\tikzstyle directed=[postaction={decorate,decoration={markings,
    mark=at position #1 with {\arrow[draw=black, line width=0.3mm]{>}}}}]
\tikzstyle rdirected=[postaction={decorate,decoration={markings,
    mark=at position #1 with {\arr1+ow[draw=black, line width=0.3mm]{<}}}}]
\tikzstyle ddirected=[postaction={decorate,decoration={markings,
    mark=at position #1 with {\draw[very thin,black,<->] (0,0) to (.2,0);}}}]
\tikzstyle{snakeline} = [decorate, decoration={pre length=0.2cm,
                         post length=0.2cm, snake, amplitude=.4mm,
                         segment length=2mm},thick]

\makeatletter
\newcommand{\setword}[2]{%
  \phantomsection
  #1\def\@currentlabel{\unexpanded{#1}}\label{#2}%
}
\makeatother

\newcommand{\placeholder}{\underline{\phantom{a}}}

\newcommand{\C}{\mathbb{C}}
\newcommand{\Z}{\mathbb{Z}}
\newcommand{\R}{\mathbb{R}}

\newcommand{\K}{\mathbbm{k}}

\newcommand{\N}{\mathbb{Z}_{\geq 0}}

\newcommand{\mstuff}[1]{\boldsymbol{#1}}
\newcommand{\varstuff}[1]{\mathsf{#1}}
\newcommand{\algstuff}[1]{\mathrm{#1}}

\newcommand{\obstuff}[1]{\mathtt{#1}}

\newcommand{\bvec}{\mathsf{C}}

\newcommand{\zigzag}{\algstuff{Z}_{\rightleftarrows}}
\newcommand{\bzigzag}{\algstuff{Z}_{\rightleftarrows}^{\bvec}}

\newcommand{\Gg}{\Gamma}
\newcommand{\Ggd}{\Gamma_{\rightleftarrows}}

\newcommand{\amatrix}{\mstuff{A}}
\newcommand{\imatrix}{\mstuff{I}}
\newcommand{\ematrix}{\mstuff{E}}
\newcommand{\cmatrix}{\mstuff{C}}
\newcommand{\ccmatrix}{\mstuff{c}}
\newcommand{\dmatrix}{\mstuff{D}}

\newcommand{\Rg}{\Theta}

\newcommand{\ii}[1][i]{\mathtt{#1}}

\newcommand{\uarrow}{\tikz[baseline=-2.5,scale=0.25]{\draw[-] (0,0) to (.75,0);}}
\newcommand{\marrow}{\tikz[baseline=-2.5,scale=0.25]{\draw[->] (0,0) to (.75,0);}}

\newcommand{\upathx}[2]{\obstuff{#1}\uarrow\obstuff{#2}}

\newcommand{\pathx}[2]{\obstuff{#1}\marrow\obstuff{#2}}
\newcommand{\pathxx}[3]{\obstuff{#1}\marrow\obstuff{#2}\marrow\obstuff{#3}}
\newcommand{\pathxxx}[4]{\obstuff{#1}\marrow\obstuff{#2}\marrow\obstuff{#3}\marrow\obstuff{#4}}
\newcommand{\pathxxxx}[5]{\obstuff{#1}\marrow\obstuff{#2}\marrow\obstuff{#3}\marrow\obstuff{#4}\marrow\obstuff{#5}}

\newcommand{\loopy}[1][s]{\alpha_{\mathtt{#1}}}
\newcommand{\volele}{\mathtt{x}}

\newcommand{\Hom}[1]{\algstuff{H}\mathrm{om}_{#1}}
\newcommand{\End}[1]{\algstuff{E}\mathrm{nd}_{#1}}

\newcommand{\typeADE}{\mathsf{A}\mathsf{D}\mathsf{E}}
\newcommand{\typeDE}{\mathsf{D}\mathsf{E}}

\newcommand{\typeA}[1][n]{\mathsf{A}_{#1}}
\newcommand{\atypeA}[1][n]{\widetilde{\mathsf{A}}_{#1}}

\newcommand{\typeD}[1][n]{\mathsf{D}_{#1}}
\newcommand{\atypeD}[1][n]{\widetilde{\mathsf{D}}_{#1}}

\newcommand{\typeE}[1][6]{\mathsf{E}_{#1}}
\newcommand{\atypeE}[1][6]{\widetilde{\mathsf{E}}_{#1}}

\newcommand{\qpar}{\varstuff{q}}
\newcommand{\qnumber}[1]{#1_{\qpar}}

\newcommand{\qdim}{\qpar\mathrm{dim}}

\newcommand{\lproj}[1][i]{\algstuff{P}_{\ii[#1]}}
\newcommand{\lprojb}[1][i]{\algstuff{P}_{\ii[#1]}^{\bvec}}
\newcommand{\rproj}[1][i]{{}_{\ii[#1]}\algstuff{P}}

\newcommand{\lsimple}[1][i]{\algstuff{L}_{\ii[#1]}}
\newcommand{\rsimple}[1][i]{{}_{\ii[#1]}\algstuff{L}}

\newcommand{\ldmod}[1][i]{\Delta_{#1}}

\newcommand{\calg}{\algstuff{R}}

\newcommand{\cset}{\obstuff{X}}
\newcommand{\cmset}{\obstuff{M}}
\newcommand{\ciset}{\obstuff{E}}
\newcommand{\coset}{\obstuff{O}}
\newcommand{\cbasis}{\obstuff{C}}

\newcommand{\ceps}{\boldsymbol{\varepsilon}}

\newcommand{\ord}[1]{<_{#1}}
\newcommand{\Ord}[1]{\leq_{#1}}

\newcommand{\invo}{\star}

\newcommand{\cbas}[2]{\mathtt{C}_{#1}^{#2}}

\newcommand{\qhalg}{\algstuff{R}}
\newcommand{\qhideal}{\algstuff{J}}

\newcommand{\kalg}{\algstuff{R}}
\newcommand{\xpar}{\varstuff{X}}
\newcommand{\xmpar}{\mstuff{X}}
\newcommand{\sshift}{s}

\newcommand{\ympar}{\mstuff{Y}}

\theoremstyle{definition}
\newtheorem{theoremm}{Theorem}[section]

\declaretheorem[style=definition,name=Theorem,qed=$\square$,numberlike=theoremm]{theorem}

\declaretheorem[style=definition,name=Lemma,qed=$\square$,numberlike=theoremm]{lemma}
\declaretheorem[style=definition,name=Lemma,qed=$\blacksquare$,numberlike=theoremm]{lemmaqed}

\declaretheorem[style=definition,name=Proposition,qed=$\square$,numberlike=theoremm]{proposition}
\declaretheorem[style=definition,name=Proposition,qed=$\blacksquare$,numberlike=theoremm]{propositionqed}

\declaretheorem[style=definition,name=Example,qed=$\blacktriangle$,numberlike=theorem]{example}
\declaretheorem[style=definition,name=Definition,qed=$\blacktriangle$,numberlike=theorem]{definition}

\declaretheorem[style=definition,name=Remark,qed=$\blacktriangle$,numberlike=theorem]{remark}

\newtheorem{theoremmain}{Theorem}
\newtheorem{theoremmainb}{Theorem}

\def\notation#1#2#3{\rlap{\hyperref[#1]{{\color{orchid}#2}}}\hspace*{8.2mm} \hbox to 47mm{#3\hfill}}

\allowdisplaybreaks

\numberwithin{equation}{section}
\usepackage[hypertexnames=false]{hyperref}
\usepackage{cleveref,bookmark}
\hypersetup{
    pdftoolbar=true,        
    pdfmenubar=true,        
    pdffitwindow=false,     
    pdfstartview={FitH},    
    pdftitle={Title},    
    pdfauthor={Names},     
    pdfsubject={},   
    pdfcreator={Names},   
    pdfproducer={Names}, 
    pdfkeywords={}, 
    pdfnewwindow=true,      
    colorlinks=true,       
    linkcolor=mydarkblue,          
    citecolor=teal,        
    filecolor=magenta,      
    urlcolor=orchid,          
    linkbordercolor=lava,
    citebordercolor=teal,
    urlbordercolor=orchid,  
    linktocpage=true
}
\let\fullref\autoref
\def\makeautorefname#1#2{\expandafter\def\csname#1autorefname\endcsname{#2}}
\makeautorefname{equation}{Equation}%
\makeautorefname{footnote}{footnote}%
\makeautorefname{item}{item}%
\makeautorefname{figure}{Figure}%
\makeautorefname{table}{Table}%
\makeautorefname{part}{Part}%
\makeautorefname{appendix}{Appendix}%
\makeautorefname{chapter}{Chapter}%
\makeautorefname{section}{Section}%
\makeautorefname{subsection}{Section}%
\makeautorefname{subsubsection}{Section}%
\makeautorefname{paragraph}{Paragraph}%
\makeautorefname{subparagraph}{Paragraph}%
\makeautorefname{theorem}{Theorem}%
\makeautorefname{theo}{Theorem}%
\makeautorefname{thm}{Theorem}%
\makeautorefname{addendum}{Addendum}%
\makeautorefname{addend}{Addendum}%
\makeautorefname{add}{Addendum}%
\makeautorefname{maintheorem}{Main theorem}%
\makeautorefname{mainthm}{Main theorem}%
\makeautorefname{corollary}{Corollary}%
\makeautorefname{corol}{Corollary}%
\makeautorefname{coro}{Corollary}%
\makeautorefname{cor}{Corollary}%
\makeautorefname{lemma}{Lemma}%
\makeautorefname{lemm}{Lemma}%
\makeautorefname{lem}{Lemma}%
\makeautorefname{sublemma}{Sublemma}%
\makeautorefname{sublem}{Sublemma}%
\makeautorefname{subl}{Sublemma}%
\makeautorefname{proposition}{Proposition}%
\makeautorefname{proposit}{Proposition}%
\makeautorefname{propos}{Proposition}%
\makeautorefname{propo}{Proposition}%
\makeautorefname{prop}{Proposition}%
\makeautorefname{property}{Property}
\makeautorefname{proper}{Property}
\makeautorefname{scholium}{Scholium}%
\makeautorefname{step}{Step}%
\makeautorefname{conjecture}{Conjecture}%
\makeautorefname{conject}{Conjecture}%
\makeautorefname{conj}{Conjecture}%
\makeautorefname{question}{Question}
\makeautorefname{questn}{Question}
\makeautorefname{quest}{Question}
\makeautorefname{ques}{Question}
\makeautorefname{qn}{Question}
\makeautorefname{definition}{Definition}%
\makeautorefname{defin}{Definition}%
\makeautorefname{defi}{Definition}%
\makeautorefname{def}{Definition}%
\makeautorefname{dfn}{Definition}%
\makeautorefname{notation}{Notation}
\makeautorefname{nota}{Notation}
\makeautorefname{notn}{Notation}
\makeautorefname{remark}{Remark}%
\makeautorefname{rema}{Remark}%
\makeautorefname{rem}{Remark}%
\makeautorefname{rmk}{Remark}%
\makeautorefname{rk}{Remark}%
\makeautorefname{remarks}{Remarks}%
\makeautorefname{rems}{Remarks}%
\makeautorefname{rmks}{Remarks}%
\makeautorefname{rks}{Remarks}%
\makeautorefname{example}{Example}%
\makeautorefname{examp}{Example}%
\makeautorefname{exmp}{Example}%
\makeautorefname{exam}{Example}%
\makeautorefname{exa}{Example}%
\makeautorefname{algorithm}{Algorithm}%
\makeautorefname{algo}{Algorithm}%
\makeautorefname{alg}{Algorithm}%
\makeautorefname{axiom}{Axiom}%
\makeautorefname{axi}{Axiom}%
\makeautorefname{ax}{Axiom}%
\makeautorefname{case}{Case}%
\makeautorefname{claim}{Claim}%
\makeautorefname{clm}{Claim}%
\makeautorefname{assumption}{Assumption}%
\makeautorefname{assumpt}{Assumption}%
\makeautorefname{conclusion}{Conclusion}%
\makeautorefname{concl}{Conclusion}%
\makeautorefname{conc}{Conclusion}%
\makeautorefname{condition}{Condition}%
\makeautorefname{condit}{Condition}%
\makeautorefname{cond}{Condition}%
\makeautorefname{construction}{Construction}%
\makeautorefname{construct}{Construction}%
\makeautorefname{const}{Construction}%
\makeautorefname{cons}{Construction}%
\makeautorefname{criterion}{Criterion}%
\makeautorefname{criter}{Criterion}%
\makeautorefname{crit}{Criterion}%
\makeautorefname{exercise}{Exercise}%
\makeautorefname{exer}{Exercise}%
\makeautorefname{exe}{Exercise}%
\makeautorefname{problem}{Problem}%
\makeautorefname{problm}{Problem}%
\makeautorefname{probm}{Problem}%
\makeautorefname{prob}{Problem}%
\makeautorefname{solution}{Solution}%
\makeautorefname{soln}{Solution}%
\makeautorefname{sol}{Solution}%
\makeautorefname{summary}{Summary}%
\makeautorefname{summ}{Summary}%
\makeautorefname{sum}{Summary}%
\makeautorefname{operation}{Operation}%
\makeautorefname{oper}{Operation}%
\makeautorefname{observation}{Observation}%
\makeautorefname{observn}{Observation}%
\makeautorefname{obser}{Observation}%
\makeautorefname{obs}{Observation}%
\makeautorefname{ob}{Observation}%
\makeautorefname{convention}{Convention}%
\makeautorefname{convent}{Convention}%
\makeautorefname{conv}{Convention}%
\makeautorefname{cvn}{Convention}%
\makeautorefname{warning}{Warning}%
\makeautorefname{warn}{Warning}%
\makeautorefname{note}{Note}%
\makeautorefname{fact}{Fact}%

\begin{document}
\vbadness=10001
\hbadness=10001
\title[Algebraic properties of zigzag algebras]{Algebraic properties of zigzag algebras}
\author[M. Ehrig and D. Tubbenhauer]{Michael Ehrig and Daniel Tubbenhauer}

\address{M.E.: Beijing Institute of Technology, School of Mathematics and Statistics, Liangxiang Campus of Beijing Institute of Technology, Fangshan District, 100488 Beijing, China}
\email{micha.ehrig@outlook.com}

\address{D.T.: Institut f\"ur Mathematik, Universit\"at Z\"urich, Winterthurerstrasse 190, Campus Irchel, Office Y27J32, CH-8057 Z\"urich, Switzerland, \href{www.dtubbenhauer.com}{www.dtubbenhauer.com}}
\email{daniel.tubbenhauer@math.uzh.ch}

\begin{abstract}
We give necessary and 
sufficient conditions for zigzag algebras 
and certain generalizations of them to be 
(relative) cellular, quasi-hereditary or Koszul.
\end{abstract}

\maketitle

\vspace*{-1cm}

\tableofcontents

\vspace*{-1cm}

\renewcommand{\theequation}{\thesection-\arabic{equation}}
\section{Introduction}\label{section:intro}

Throughout, we let $\Gg$ denote 
a finite, connected, simple graph, and we work over 
an arbitrary field $\K$.
Let $\zigzag=\zigzag(\Gg)$ be the zigzag algebra 
associated to $\Gg$, and let further $\bzigzag=\bzigzag(\Gg)$ be the 
zigzag algebra with a vertex-loop condition (vertex condition for short)
at some fixed set of vertices $\bvec\neq\emptyset$.

\subsubsection*{The main statements}\label{subsubsec:main}
Zigzag algebras are around for many years, see 
e.g. \cite{w-zigzag} for an early reference 
from classical algebra. Further,
as shown in e.g. \cite{hk-zigzag}, 
they appear in various places 
in modern mathematics. For example, for a viewpoint from 
categorical braid group actions see \cite{ks-quivers} 
and \cite{kms-brief-cat}, for one from symplectic 
geometry see \cite{el-zigzag-symplectic}, for one 
from KLR algebras, associated representations 
and blocks of symmetric groups see 
\cite{km-affine-zigzag} and \cite{ek-zigzag-symmetric}, 
for one from geometric group theory see \cite{li-freebraid}, 
and for one from Soergel bimodules and $2$-representation 
theory see \cite{mt-soergel}. Similarly, the algebra 
$\bzigzag$ comes from considerations in
modular representation theory or representation theory 
at roots of unity, see e.g. \cite{at-tilting}, 
\cite{qs-cat-burau}, or versions of category $\mathcal{O}$, see 
e.g. \cite{st-endo-cat-o}, \cite{ck-arc} or \cite{bs-arc4}.

The purpose of this 
paper is to show the following algebraic 
properties of $\zigzag$ and $\bzigzag$.

\begin{theoremmain}\label{theorem:cell}
$\zigzag$ is cellular if and only if
$\Gg$ is a finite type $\typeA[]$ graph.
$\zigzag$ is relative cellular if and only if
$\Gg$ is a finite or affine type $\typeA[]$ graph.
\end{theoremmain}

For us e.g. not being cellular 
always means that there is no choice of a cell datum.
Further, in all cases where $\zigzag$ is 
(relative) cellular, the path length
endows it with the structure of a graded 
(relative) cellular algebra.

\begin{theoremmain}\label{theorem:qh}
$\zigzag$ is never quasi-hereditary.
\end{theoremmain}

\begin{theoremmain}\label{theorem:koszul}
$\zigzag$ is Koszul if and only if $\Gg$ is not a 
type $\typeADE$ graph.
\end{theoremmain}

Additionally, we give an algorithmic construction for the minimal linear projective resolutions of simple 
$\zigzag$-modules.

Using the same ideas 
as for $\zigzag$ we can also prove:

\begin{theoremmainb}\label{theorem:cellb}
$\bzigzag$ is cellular if and only if
$\Gg$ is a finite type $\typeA[]$ graph and the 
vertex condition is imposed on one leaf.
$\bzigzag$ is relative cellular in exactly the same cases.
\end{theoremmainb}

\begin{theoremmainb}\label{theorem:qhb}
$\bzigzag$ is quasi-hereditary if and only 
if $\Gg$ is a finite type $\typeA[]$ graph and the 
vertex condition is imposed on one leaf.
\end{theoremmainb}

\begin{theoremmainb}\label{theorem:koszulb}
$\bzigzag$ is always Koszul.
\end{theoremmainb}

Note that, as follows from our main 
theorems and their proofs, $\bzigzag(\typeA[])$ with vertex condition imposed 
on one leaf is a graded cellular, quasi-hereditary, Koszul algebra, 
which makes it quite special, since these are the only 
zigzag algebras having these properties. Indeed, 
$\bzigzag(\typeA[])$ has 
some very nice spectral properties 
(in the sense of spectral graph theory), and is the version of 
zigzag algebras which is mostly studied in the literature.

\begin{remark}\label{remark:generalization-graph}
Let us note that the restrictions on $\Gg$ being finite, connected and simple can be relaxed, 
and we imposed them for convenience.
\end{remark}

\subsubsection*{Strengthening and collecting known results}\label{subsubsec:main-extra}

Some of these results are actually 
known and stated implicitly or explicitly in the literature.
(This is partially due to the different nomenclatures.) 
So another point of our paper is to collect these in one document.

To elaborate a bit, the construction of the 
cellular respectively relative cellular 
structures for \fullref{theorem:cell} 
are folklore or standard, while the converses appear to be new.
\fullref{theorem:qh} is very easy 
(indeed, our proof here is basically one line) and 
stated for completeness. \fullref{theorem:koszul} appears in 
\cite{mv-koszul} or \cite{ee-koszul}, where the theorem is proven 
for the preprojective algebra (which is the Koszul 
dual of the zigzag algebra 
in case $\Gg$ is bipartite), 
and in \cite{d-zigzag-koszul} whenever $\Gg$ has a circle.
However, our methods (inspired by \cite{b-cube-zero}) 
to prove \fullref{theorem:koszul}, constructing explicit 
resolutions using Chebyshev polynomials, seem to be new as well.
Having this, let us note that the well-established periodicity of the 
Chebyshev polynomials for type $\typeADE$ graphs 
can be used to reprove the results about almost Koszulness 
of the preprojective algebra in these cases \cite{bbk-almost-koszul}.

\makeautorefname{theoremmainb}{Theorems}

Finally, as far as we can tell, \fullref{theorem:cellb}, \ref{theorem:qhb} and 
\ref{theorem:koszulb} are new and actually generalizations 
of the other three theorems. But, again, the 
constructions of the relevant structures for 
\fullref{theorem:cellb} and \ref{theorem:qhb} are folklore, 
and the punchlines are the conserves.

\begin{remark}\label{remark:generalization}
Note that we will always assume $\bvec\neq\emptyset$. The 
only reason for this is that we want to treat 
our main statements separately, i.e. there is no problem 
to allow $\bvec=\emptyset$ in most of the arguments which we are going to use, 
and $\bzigzag$ is in fact a strict generalization of $\zigzag=\zigzag^{\emptyset}$. 
And although we will formulate $\bzigzag$ 
as a quotient of $\zigzag$, we think of it as a quasi-hereditary 
cover of $\zigzag$, since this is what happens in the case 
of type $\typeA[]$ graphs.
\end{remark}

\makeautorefname{theoremmainb}{Theorem} 

\subsubsection*{Towards generalizations}\label{subsubsec:future}

A striking question is how zigzag algebras can be generalized, 
and how to control them algebraically.

There are at least two different 
generalizations of zigzag algebras for which our methods 
seems to be applicable: either the one from \cite{g-higher-zigzag}, 
generalizing the connection to e.g. the preprojective algebras, Iyama's higher 
representation types and categorical group actions, 
or the one from \cite[Section 5C]{mmmt-trihedral}, generalizing 
the connections to e.g. $2$-representation 
theory, affine Hecke algebras and modular representation theory. 
In both cases the spectral properties of the underlying 
graphs seem to play a crucial role and we hope that our methods 
presented here generalize to those algebras. 
In particular, the generalizations of zigzag algebras 
from \cite[Section 5C]{mmmt-trihedral} are usually not 
connected to preprojective algebras, and their quasi-hereditary 
covers are similar in spirit to $\bzigzag$ (by adding certain 
vertex conditions), and our paper might help to 
understand algebraic properties of these algebras.
\medskip

\noindent\textbf{Acknowledgments.} We like to thank 
Kevin Coulembier, Marco Mackaay, Volodymyr 
Mazorchuk and Krzysztof Putyra for helpful discussions, conversations and 
exchange of emails, Catharina Stroppel for comments 
on a draft of this paper, and the referee for helpful 
suggestions. We like to thank SAGEMath, which 
should have been a coauthor of this paper, but modestly refused to.

M.E. was supported by the Australian Research Council Grant DP150103431 and
D.T. was partially supported by NCCR SwissMAP during this work.
%
\section{Preliminaries}\label{section:prelim}

We denote by 
$\ii,\ii[j]$ etc. the vertices of $\Gg$, and $\upathx{i}{j}$ means that 
$\ii$ and $\ii[j]$ are connected in $\Gg$ by an edge. 
For each such graph $\Gg$ we chose an enumeration of its 
vertices, and we obtain its adjacency matrix $\amatrix=\amatrix(\Gg)$.

\begin{example}\label{example:ADE}
Of paramount importance for us are the finite 
type $\typeADE$ graphs
\begin{gather*}
\,
\xy
(0,0)*{
\begin{tikzcd}[ampersand replacement=\&,row sep=large,column sep=tiny,arrows={shorten >=-.5ex,shorten <=-.5ex},labels={inner sep=.05ex}]
\ii[1]
\arrow[thick,-]{r}{}
\&
\ii[2]
\arrow[thick,-]{r}{}
\&
\cdots
\arrow[thick,-]{r}{}
\&
\ii[n{-}1]
\arrow[thick,-]{r}{}
\&
\ii[n]
\end{tikzcd}};
(0,-10)*{\text{type $\typeA$; $n\in\Z_{\geq 1}$}};
\endxy,
\quad
\xy
(0,0)*{
\begin{tikzcd}[ampersand replacement=\&,row sep=tiny,column sep=tiny,arrows={shorten >=-.5ex,shorten <=-.5ex},labels={inner sep=.05ex}]
\phantom{.}
\&
\phantom{.}
\&
\phantom{.}
\&
\phantom{.}
\&
\ii[n{-}1]
\\
\ii[1]
\arrow[thick,-]{r}{}
\&
\ii[2]
\arrow[thick,-]{r}{}
\&
\cdots
\arrow[thick,-]{r}{}
\&
\ii[n{-}2]
\arrow[thick,-]{rd}{}
\arrow[thick,-]{ru}{}
\&
\phantom{1}
\\
\phantom{.}
\&
\phantom{.}
\&
\phantom{.}
\&
\phantom{.}
\&
\ii[n]\phantom{-}\phantom{1}
\end{tikzcd}};
(0,-10)*{\text{type $\typeD$; $n\in\Z_{\geq 4}$}};
(0,10)*{\text{\tiny $\amatrix(\typeD[4])=\begin{psmallmatrix}0 & 1 & 0 & 0\\ 
1 & 0 & 1 & 1\\ 
0 & 1 & 0 & 0\\ 
0 & 1 & 0 & 0\end{psmallmatrix}$}};
\endxy,
\end{gather*}
\begin{gather*}
\,
\xy
(0,0)*{
\begin{tikzcd}[ampersand replacement=\&,row sep=tiny,column sep=tiny,arrows={shorten >=-.5ex,shorten <=-.5ex},labels={inner sep=.05ex}]
\phantom{.}
\&
\phantom{.}
\&
\ii[6]
\&
\phantom{.}
\&
\phantom{.}
\\
\ii[1]
\arrow[thick,-]{r}{}
\&
\ii[2]
\arrow[thick,-]{r}{}
\&
\ii[3]
\arrow[thick,-]{r}{}
\arrow[thick,-]{u}{}
\&
\ii[4]
\arrow[thick,-]{r}{}
\&
\ii[5]
\end{tikzcd}};
(0,-10)*{\text{type $\typeE[6]$}};
\endxy,
\quad
\xy
(0,0)*{
\begin{tikzcd}[ampersand replacement=\&,row sep=tiny,column sep=tiny,arrows={shorten >=-.5ex,shorten <=-.5ex},labels={inner sep=.05ex}]
\phantom{.}
\&
\phantom{.}
\&
\ii[7]
\&
\phantom{.}
\&
\phantom{.}
\&
\phantom{.}
\\
\ii[1]
\arrow[thick,-]{r}{}
\&
\ii[2]
\arrow[thick,-]{r}{}
\&
\ii[3]
\arrow[thick,-]{r}{}
\arrow[thick,-]{u}{}
\&
\ii[4]
\arrow[thick,-]{r}{}
\&
\ii[5]
\arrow[thick,-]{r}{}
\&
\ii[6]
\end{tikzcd}};
(0,-10)*{\text{type $\typeE[7]$}};
\endxy,
\quad
\xy
(0,0)*{
\begin{tikzcd}[ampersand replacement=\&,row sep=tiny,column sep=tiny,arrows={shorten >=-.5ex,shorten <=-.5ex},labels={inner sep=.05ex}]
\phantom{.}
\&
\phantom{.}
\&
\ii[8]
\&
\phantom{.}
\&
\phantom{.}
\&
\phantom{.}
\&
\phantom{.}
\\
\ii[1]
\arrow[thick,-]{r}{}
\&
\ii[2]
\arrow[thick,-]{r}{}
\&
\ii[3]
\arrow[thick,-]{r}{}
\arrow[thick,-]{u}{}
\&
\ii[4]
\arrow[thick,-]{r}{}
\&
\ii[5]
\arrow[thick,-]{r}{}
\&
\ii[6]
\arrow[thick,-]{r}{}
\&
\ii[7]
\end{tikzcd}};
(0,-10)*{\text{type $\typeE[8]$}};
\endxy,
\end{gather*}
as well as their affine counterparts
\begin{gather*}
\,
\xy
(0,0)*{
\begin{tikzcd}[ampersand replacement=\&,row sep=tiny,column sep=tiny,arrows={shorten >=-.5ex,shorten <=-.5ex},labels={inner sep=.05ex}]
\phantom{.}
\&
\ii[1]
\arrow[thick,-]{dr}{}
\&
\phantom{.}
\\
\phantom{.}\ii[0]\phantom{.}
\arrow[thick,-]{ur}{}
\&
\phantom{.}
\&
\cdots
\arrow[thick,-]{dl}{}
\\
\phantom{.}
\&
\ii[n]
\arrow[thick,-]{ul}{}
\&
\phantom{.}
\end{tikzcd}};
(0,-10)*{\text{type $\atypeA$; $n\in\Z_{\geq 2}$}};
(0,10)*{\text{\tiny $\amatrix(\atypeA[2])=\begin{psmallmatrix}0 & 1 & 1\\ 
1 & 0 & 1\\ 
1 & 1 & 0\end{psmallmatrix}$}};
\endxy,
\quad
\xy
(0,0)*{
\begin{tikzcd}[ampersand replacement=\&,row sep=tiny,column sep=tiny,arrows={shorten >=-.5ex,shorten <=-.5ex},labels={inner sep=.05ex}]
\ii[0]
\&
\phantom{.}
\&
\phantom{.}
\&
\phantom{.}
\&
\phantom{.}
\&
\ii[n{-}1]
\\
\phantom{.}
\&
\ii[2]
\arrow[thick,-]{r}{}
\arrow[thick,-]{lu}{}
\arrow[thick,-]{ld}{}
\&
\ii[3]
\arrow[thick,-]{r}{}
\&
\cdots
\arrow[thick,-]{r}{}
\&
\ii[n{-}2]
\arrow[thick,-]{rd}{}
\arrow[thick,-]{ru}{}
\&
\phantom{1}
\\
\ii[1]
\&
\phantom{.}
\&
\phantom{.}
\&
\phantom{.}
\&
\phantom{.}
\&
\ii[n]\phantom{-}\phantom{1}
\end{tikzcd}};
(0,-10)*{\text{type $\atypeD$; $n\in\Z_{\geq 4}$}};
\endxy,
\end{gather*}
\begin{gather*}
\,
\xy
(0,0)*{
\begin{tikzcd}[ampersand replacement=\&,row sep=tiny,column sep=tiny,arrows={shorten >=-.5ex,shorten <=-.5ex},labels={inner sep=.05ex}]
\phantom{.}
\&
\phantom{.}
\&
\ii[0]
\&
\phantom{.}
\&
\phantom{.}
\\
\phantom{.}
\&
\phantom{.}
\&
\ii[6]
\arrow[thick,-]{u}{}
\&
\phantom{.}
\&
\phantom{.}
\\
\ii[1]
\arrow[thick,-]{r}{}
\&
\ii[2]
\arrow[thick,-]{r}{}
\&
\ii[3]
\arrow[thick,-]{r}{}
\arrow[thick,-]{u}{}
\&
\ii[4]
\arrow[thick,-]{r}{}
\&
\ii[5]
\end{tikzcd}};
(0,-11)*{\text{type $\atypeE[6]$}};
\endxy,
\;
\xy
(0,0)*{
\begin{tikzcd}[ampersand replacement=\&,row sep=tiny,column sep=tiny,arrows={shorten >=-.5ex,shorten <=-.5ex},labels={inner sep=.05ex}]
\phantom{.}
\&
\phantom{.}
\&
\phantom{.}
\&
\phantom{0}
\&
\phantom{.}
\&
\phantom{.}
\&
\phantom{.}
\\
\phantom{.}
\&
\phantom{.}
\&
\phantom{.}
\&
\ii[7]
\&
\phantom{.}
\&
\phantom{.}
\&
\phantom{.}
\\
\ii[0]
\arrow[thick,-]{r}{}
\&
\ii[1]
\arrow[thick,-]{r}{}
\&
\ii[2]
\arrow[thick,-]{r}{}
\&
\ii[3]
\arrow[thick,-]{r}{}
\arrow[thick,-]{u}{}
\&
\ii[4]
\arrow[thick,-]{r}{}
\&
\ii[5]
\arrow[thick,-]{r}{}
\&
\ii[6]
\end{tikzcd}};
(0,-11)*{\text{type $\atypeE[7]$}};
\endxy,
\;
\xy
(0,0)*{
\begin{tikzcd}[ampersand replacement=\&,row sep=tiny,column sep=tiny,arrows={shorten >=-.5ex,shorten <=-.5ex},labels={inner sep=.05ex}]
\phantom{.}
\&
\phantom{.}
\&
\phantom{.}
\&
\phantom{0}
\&
\phantom{.}
\&
\phantom{.}
\&
\phantom{.}
\&
\phantom{.}
\\
\phantom{.}
\&
\phantom{.}
\&
\ii[8]
\&
\phantom{.}
\&
\phantom{.}
\&
\phantom{.}
\&
\phantom{.}
\&
\phantom{.}
\\
\ii[1]
\arrow[thick,-]{r}{}
\&
\ii[2]
\arrow[thick,-]{r}{}
\&
\ii[3]
\arrow[thick,-]{r}{}
\arrow[thick,-]{u}{}
\&
\ii[4]
\arrow[thick,-]{r}{}
\&
\ii[5]
\arrow[thick,-]{r}{}
\&
\ii[6]
\arrow[thick,-]{r}{}
\&
\ii[7]
\arrow[thick,-]{r}{}
\&
\ii[0]
\end{tikzcd}};
(0,-11)*{\text{type $\atypeE[8]$}};
\endxy.
\end{gather*}
The enumeration of the vertices matters for some calculations, 
and we always number them as indicated above. 
(Note that we omit the type $\atypeA[1]$ graph.)
\end{example}

\subsubsection{The zigzag algebra}\label{subsubsec:zigzag}

The \textit{double graph} $\Ggd$ is the oriented graph obtained from $\Gg$ by doubling 
all edges $\upathx{i}{j}$ of $\Gg$ into a pair of parallel edges 
$\pathx{i}{j}$ (oriented from $\ii$ to $\ii[j]$) and $\pathx{j}{i}$ 
(oriented from $\ii[j]$ to $\ii$), and by adding two loops 
$\loopy=(\loopy)_{\ii}$ and $\loopy[t]=(\loopy[t])_{\ii}$ 
per vertex.

Let $\algstuff{R}(\Ggd)$ denote the path algebra for 
$\Ggd$. It is graded by using the path length, but putting loops in 
degree $2$. We identify its length 
zero paths with the vertices of $\Gg$ (e.g. $\ii$ also denotes 
the vertex idempotent), and we let 
$\pathxx{i}{j}{k}=\pathx{i}{j}\circ\pathx{j}{k}$ etc. denote the 
composition.

\begin{definition}\label{definition:zigzag}
Let $\zigzag=\zigzag(\Gg)$, for $\Gg$ having at least three vertices, 
be the quotient of $\algstuff{R}(\Ggd)$ by the following 
defining relations.

\begin{enumerate}[label=(\alph*)]

\setlength\itemsep{.15cm}

\renewcommand{\theenumi}{(\ref{definition:zigzag}.a)}
\renewcommand{\labelenumi}{\theenumi}

\item \label{enum:zigzag-1} \textbf{Boundedness.} Any path involving three 
distinct vertices is zero.

\renewcommand{\theenumi}{(\ref{definition:zigzag}.b)}
\renewcommand{\labelenumi}{\theenumi}

\item \label{enum:zigzag-2} \textbf{The relations of the cohomology ring of 
the variety of full flags in $\C^{2}$.} $\loopy\circ\loopy[t]=\loopy[t]\circ\loopy$, 
$\loopy+\loopy[t]=0$ and $\loopy\circ\loopy[t]=0$.

\renewcommand{\theenumi}{(\ref{definition:zigzag}.c)}
\renewcommand{\labelenumi}{\theenumi}

\item \label{enum:zigzag-3} \textbf{Zigzag.} $\pathxx{i}{j}{i}=\loopy-\loopy[t]$ 
for $\upathx{i}{j}$.

\end{enumerate}
In case $\Gg$ has one vertex we let $\zigzag=\K[\loopy,\loopy[t]]/(\loopy+\loopy[t],\loopy\loopy[t])$, 
by convention, and in case $\Gg$ has two vertices we additionally to \ref{enum:zigzag-1}, 
\ref{enum:zigzag-2} and \ref{enum:zigzag-3} kill paths of length three.

We call $\zigzag$ the \textit{zigzag algebra} associated to $\Gg$.
\end{definition}

The relations of $\zigzag$ are homogeneous with 
respect to the path length grading, which thus endow $\zigzag$ 
with the structure of a graded algebra. (Throughout, graded means $\Z$-graded.)

\begin{remark}\label{remark:zigzag-usual}
Our definition of $\zigzag$ is slightly different from the 
one as e.g. in \cite[Section 3]{hk-zigzag}, but they are equivalent. 
To see this use the isomorphism from $\K[\volele]/(\volele^{2})$ to 
$\K[\loopy,\loopy[t]]/(\loopy+\loopy[t],\loopy\loopy[t])$ 
given by $\volele\mapsto\loopy-\loopy[t]\in\End{\zigzag}(\ii)$. (If $\K$ has characteristic 
$2$, then $\volele\mapsto\loopy=\loopy[t]$ gives the isomorphism.) We prefer 
the formulation as in \fullref{definition:zigzag} since it fits 
to the generalizations of the zigzag algebra 
from \cite[Section 5C]{mmmt-trihedral}.
\end{remark}

We call $\volele_{\ii}=\loopy-\loopy[t]$ 
(or $\volele_{\ii}=\loopy$ in characteristic $2$) the \textit{volume element} (at vertex $\ii$).

\begin{example}\label{example:typeA}
The most classical examples of zigzag algebras are the cases where 
$\Gg$ is either a type $\typeA$ graph or a type $\atypeA$ graph.
\begin{gather*}
\zigzag(\typeA)=
\xy
(0,0)*{
\begin{tikzcd}[ampersand replacement=\&,row sep=large,column sep=scriptsize,arrows={shorten >=-.5ex,shorten <=-.5ex},labels={inner sep=.05ex}]
\ii[1]
\arrow[out=60,in=120,loop,distance=1.25em,swap]{}{\loopy} 
\arrow[out=-60,in=-120,loop,distance=1.25em]{}{\loopy[t]}
\arrow[yshift=.5ex,<-]{r}{}
\arrow[yshift=-.5ex,->,swap]{r}{}
\&
\ii[2]
\arrow[out=60,in=120,loop,distance=1.25em,swap]{}{\loopy} 
\arrow[out=-60,in=-120,loop,distance=1.25em]{}{\loopy[t]}
\arrow[yshift=.5ex,<-]{r}{}
\arrow[yshift=-.5ex,->,swap]{r}{}
\&
\cdots
\arrow[yshift=.5ex,<-]{r}{}
\arrow[yshift=-.5ex,->,swap]{r}{}
\&
\ii[n{-}1]
\arrow[out=60,in=120,loop,distance=1.25em,swap]{}{\loopy} 
\arrow[out=-60,in=-120,loop,distance=1.25em]{}{\loopy[t]}
\arrow[yshift=.5ex,<-]{r}{g}
\arrow[yshift=-.5ex,->,swap]{r}{f}
\&
\phantom{1}\ii[n]\phantom{1}
\arrow[out=60,in=120,loop,distance=1.25em,swap]{}{\loopy} 
\arrow[out=-60,in=-120,loop,distance=1.25em]{}{\loopy[t]}
\\
\end{tikzcd}};
(0,14)*{\text{{\tiny $f=(\pathx{\,n{-}1\,}{\,n\,}),\quad g=(\pathx{\,n\,}{\,n{-}1\,})$}}};
(0,-10)*{\text{{\tiny living on the type $\typeA$ graph}}};
\endxy
,\quad\quad
\zigzag(\atypeA)=
\xy
(0,0)*{
\begin{tikzcd}[ampersand replacement=\&,row sep=small,column sep=scriptsize,arrows={shorten >=-.5ex,shorten <=-.5ex},labels={inner sep=.05ex}]
\phantom{.}
\&
\ii[1]
\arrow[xshift=.3ex,yshift=.3ex,<-]{dr}{}
\arrow[xshift=-.3ex,yshift=-.3ex,->,swap]{dr}{}
\&
\phantom{.}
\\
\phantom{.}\ii[0]\phantom{.}
\arrow[xshift=-.3ex,yshift=.3ex,<-]{ur}{}
\arrow[xshift=.3ex,yshift=-.3ex,->,swap]{ur}{}
\&
\phantom{.}
\&
\cdots
\arrow[xshift=.3ex,yshift=-.3ex,<-]{dl}{}
\arrow[xshift=-.3ex,yshift=.3ex,->,swap]{dl}{}
\\
\phantom{.}
\&
\ii[n]
\arrow[xshift=-.3ex,yshift=-.3ex,<-]{ul}{g}
\arrow[xshift=.3ex,yshift=.3ex,->,swap]{ul}{f}
\&
\phantom{.}
\\
\end{tikzcd}};
(0,14)*{\text{{\tiny $f=(\pathx{\,n\,}{\,0\,}),\quad g=(\pathx{\,0\,}{\,n\,})$}}};
(0,-10)*{\text{{\tiny living on the type $\atypeA$ graph}}};
\endxy,
\end{gather*}
where we omitted the loops for the type $\atypeA$ graph for illustration purposes only.
\end{example}

\subsubsection{Basic properties}\label{subsubsec:zigzag-basics}

We omit the (easy) proofs of the following basics facts.

\begin{lemmaqed}\label{lemma:basics1}
If $\Gg$ has three or more vertices, then
$\zigzag$ is quadratic, i.e. it is generated in degree $1$, 
and its relations are generated in degree $2$.
\end{lemmaqed}

Note hereby that 
the relation $\loopy+\loopy[t]=0$ can be omitted by reparametrization of 
the endomorphisms spaces, cf. \fullref{remark:zigzag-usual}.

\begin{lemmaqed}\label{lemma:basics2}
The association
\[
\mathrm{tr}(\ii)=0,
\quad
\mathrm{tr}(\volele_{\ii})=1,
\quad
\mathrm{tr}(\pathx{i}{j})=0,
\]
gives rise to 
a non-degenerate trace form $\mathrm{tr}\colon\zigzag\to\K$, which endows 
$\zigzag$ with the structure of a (non-commutative) Frobenius algebra.
\end{lemmaqed}

Using \fullref{lemma:basics2}, the following is also easy to prove, where 
we recall that elements of $\zigzag$ are graded by path length.

\begin{lemmaqed}\label{lemma:basics3}
We have
\begin{gather*}
\qdim(\Hom{\zigzag}(\ii,\ii[j]))
=
\begin{cases}
\begin{aligned}
\qnumber{2}, \quad &\text{if }\ii=\ii[j], \quad &\{\ii,\volele_{\ii}\}\text{ is a basis},
\\
\qpar, \quad &\text{if }\upathx{i}{j}, \quad &\{\pathx{i}{j}\}\text{ is a basis},
\\
0, \quad &\text{else}, \quad &\emptyset\text{ is a basis},
\end{aligned}
\end{cases}
\end{gather*}
where $\qdim(\placeholder)$ denotes the graded dimension, and $\qnumber{2}=1+\qpar^{2}$.
\end{lemmaqed}

\subsubsection{Projective and simple modules}\label{subsubsec:zigzag-modules}

For $\sshift\in\Z$ we will denote by $\qpar^{\sshift}$ degree shifts using the convention that 
a map of degree $d$ between $\algstuff{M}$ and $\algstuff{N}$ is of degree 
$d-\sshift_{1}+\sshift_{2}$ seen 
as a map from $\qpar^{\sshift_{1}}\algstuff{M}$ to 
$\qpar^{\sshift_{2}}\algstuff{N}$. Further,
for us the action of $\zigzag$ on modules will always be given by
\[
\text{left: pre-composition of paths}
\quad\&\quad
\text{right: post-composition of paths}.
\] 
In particular,
\begin{gather*}
\qpar^{\sshift}\lproj
=
\{
\ii,\pathx{j}{i},\volele_{\ii}\mid\upathx{i}{j}
\},
\quad\quad
\qpar^{\sshift}\rproj
=
\{
\ii,\pathx{i}{j},\volele_{\ii}\mid\upathx{i}{j}
\},
\end{gather*}
are left, respectively right, graded
projective $\zigzag$-modules with $\ii$ being in degree $\sshift$. 
Moreover,  
\begin{gather*}
\{
\qpar^{\sshift}\lproj\mid\ii\in\Gg,\sshift\in\Z
\},
\quad\quad
\{
\qpar^{\sshift}\rproj\mid\ii\in\Gg,\sshift\in\Z
\},
\end{gather*}
are complete, irredundant sets of indecomposable, graded projective left, respectively 
right, $\zigzag$-modules.
In contrast, the simple left $\zigzag$-modules $\lsimple$ are one-dimensional and only 
one vertex idempotent of $\zigzag$ acts non-trivially on them, 
and elements of positive degree act as zero. Thus, they 
can be identified with elements of $\zigzag$ and, for example, for $\upathx{i}{j}$ 
\begin{gather*}
\lsimple
=
\{
\ii
\}
\cong
\qpar^{-1}
\{
\pathx{i}{j}
\}
\cong
\qpar^{-2}
\{
\volele_{\ii}
\}
\end{gather*}
is the same incarnation of the simple $\lsimple$ corresponding to $\ii$. Similarly, 
of course, for the right simples $\rsimple$.
Thus, we get the Loewy picture
\begin{gather}\label{eq:loewy}
\lproj
=
\begin{gathered}
\ii
\\
\pathx{\ii[j]}{\ii}
\\
\volele_{\ii}
\end{gathered}
\;(\text{for }\upathx{i}{j})
,\quad\quad
\rproj
=
\begin{gathered}
\ii
\\
\pathx{\ii}{\ii[j]}
\\
\volele_{\ii}
\end{gathered}
\;(\text{for }\upathx{i}{j})
\end{gather}
of the projectives, where the 
vertex idempotent spans the head and
the volume element spans the socle. 

Having all this, the following easy, but crucial, statement is immediate, 
where $\imatrix$ is the identity matrix. 
(Recall hereby that the \textit{graded Cartan matrix} encodes the graded filtration of 
the projectives $\lproj$ or $\rproj$ by simples, where we 
enumerate the rows and columns as given by 
the enumeration of the vertices.)

\begin{propositionqed}\label{proposition:cartan-zigzag}
The graded Cartan matrix $\cmatrix_{\qpar}=\cmatrix_{\qpar}(\zigzag)$ of $\zigzag$ is 
$\qnumber{2}\imatrix+\qpar\amatrix$.
\end{propositionqed}

In particular, forgetting the grading, the \textit{Cartan 
matrix} $\cmatrix=\cmatrix(\zigzag)$ of $\zigzag$ is just $2\imatrix+\amatrix$.

From now on we will focus on the case of left modules 
(and omit to say so); the case of right modules can be 
done in the same way.

\subsubsection{Vertex conditions}\label{subsubsec:zigzag-boundary}

Fix a non-empty set $\bvec$ of vertices of $\Gg$.

\begin{definition}\label{definition:bzigzag}
The \textit{zigzag algebra} $\bzigzag=\bzigzag(\Gg)$ \textit{with vertex condition} for $\bvec$ 
is the quotient of $\zigzag$ obtained by 
killing the volume elements $\volele_{\ii[c]}$ 
for $\ii[c]\in\bvec$.
\end{definition}

Clearly, everything done above for $\zigzag$ 
works, mutatis mutandis, for $\bzigzag$ as well. (However, $\bzigzag$ 
is always quadratic.)
In particular, the Loewy pictures are as in \eqref{eq:loewy}, but with 
the projective $\lprojb[c]$ for $\ii[c]\in\bvec$ having no volume element, and
we will use the superscript ${}^{\bvec}$ to indicate 
$\bzigzag$-modules which are different from their $\zigzag$-counterparts.

The following combinatorial difference will play a key role, where
we denote by $\ematrix_{\bvec}$ the matrix with 
only non-zero entry equal to $1$ in the $\ii[c]$-$\ii[c]$ position 
for $\ii[c]\in\bvec$.

\begin{propositionqed}\label{proposition:cartan-bzigzag}
The graded Cartan matrix $\cmatrix_{\qpar}^{\bvec}=\cmatrix_{\qpar}(\bzigzag)$ of $\bzigzag$ is 
$\qnumber{2}\imatrix-\qpar^{2}\ematrix_{\bvec}+\qpar\amatrix$.
\end{propositionqed}

The Cartan matrix $\cmatrix^{\bvec}=\cmatrix^{\bvec}(\bzigzag)$ of $\zigzag$ is
just the dequantization $\cmatrix^{\bvec}=2\imatrix-\ematrix_{\bvec}+\amatrix$.

Let $\Gg-\ii[c]$ denote the graph obtained from $\Gg$ by removing a fixed
vertex $\ii[c]$. Further, let us 
write $\zigzag=\zigzag^{\emptyset}$ for convenience of notation. 

\begin{lemma}\label{lemma:the-cartan-cell2}
We have 
$\mathrm{det}(\cmatrix_{\qpar}^{\bvec})=\mathrm{det}(\cmatrix_{\qpar}^{\bvec}(\zigzag^{\bvec-\ii[c]}))-
\mathrm{det}(\cmatrix_{\qpar}^{\bvec}(\zigzag^{\bvec-\ii[c]}(\Gg-\ii[c])))$.
\end{lemma}

Note that \fullref{lemma:the-cartan-cell2} gives a recursive way to compute the 
\textit{graded Cartan determinant} $\mathrm{det}(\cmatrix_{\qpar}^{\bvec})$ 
of $\bzigzag$ from that of $\zigzag$.
Explicitly, in case $\bvec$ has just one entry $\ii[c]$, then 
$\mathrm{det}(\cmatrix_{\qpar}^{\bvec})=\mathrm{det}(\cmatrix_{\qpar})-
\mathrm{det}(\cmatrix_{\qpar}(\zigzag(\Gg-\ii[c])))$.

\begin{proof}
By using \fullref{proposition:cartan-bzigzag}, this follows
directly by row expansion.
\end{proof}

\begin{example}\label{example:btypeA}
Very similar to \fullref{example:typeA}, the most important 
example is the case of a type $\typeA[n]$ graph with vertex condition imposed 
on one of its leafs.
\begin{gather*}
\bzigzag(\typeA[n])=
\xy
(0,0)*{
\begin{tikzcd}[ampersand replacement=\&,row sep=large,column sep=scriptsize,arrows={shorten >=-.5ex,shorten <=-.5ex},labels={inner sep=.05ex}]
\ii[1]
\arrow[yshift=.5ex,<-]{r}{}
\arrow[yshift=-.5ex,->,swap]{r}{}
\&
\ii[2]
\arrow[out=60,in=120,loop,distance=1.25em,swap]{}{\loopy} 
\arrow[out=-60,in=-120,loop,distance=1.25em]{}{\loopy[t]}
\arrow[yshift=.5ex,<-]{r}{}
\arrow[yshift=-.5ex,->,swap]{r}{}
\&
\cdots
\arrow[yshift=.5ex,<-]{r}{}
\arrow[yshift=-.5ex,->,swap]{r}{}
\&
\ii[n{-}1]
\arrow[out=60,in=120,loop,distance=1.25em,swap]{}{\loopy} 
\arrow[out=-60,in=-120,loop,distance=1.25em]{}{\loopy[t]}
\arrow[yshift=.5ex,<-]{r}{}
\arrow[yshift=-.5ex,->,swap]{r}{}
\&
\phantom{1}\ii[n]\phantom{1}
\arrow[out=60,in=120,loop,distance=1.25em,swap]{}{\loopy} 
\arrow[out=-60,in=-120,loop,distance=1.25em]{}{\loopy[t]}
\\
\end{tikzcd}};
(0,14)*{\text{{\tiny $\bvec=\{\ii[1]\}$}}};
(0,-10)*{\text{{\tiny living on the type $\typeA[n]$ graph}}};
\endxy,
\end{gather*}
where we have illustrated the case $\bvec=\{\ii[1]\}$.
\end{example}
%
\section{Cellularity}\label{section:cell}

\subsubsection{A brief reminder}\label{subsubsec:rel-cell}

We briefly recall the definition of a relative cellular algebra 
as it appears in \cite[Definition 2.1]{et-relcell}, sneaking in the 
graded setting as in \cite[Definition 2.1]{hm-klr-basis}.

\begin{definition}\label{definition:cell-algebra}
A \textit{relative cellular algebra} is an associative, unital 
algebra $\calg$ together with a (relative) cell datum, i.e.
\begin{gather*}
(\cset,\cmset,\cbasis,
{}^{\invo},
\ciset,\coset,\epsilon),
\end{gather*}
such that the following hold.
\smallskip
\begin{enumerate}

\setlength\itemsep{.15cm}

\renewcommand{\theenumi}{(\ref{definition:cell-algebra}.a)}
\renewcommand{\labelenumi}{\theenumi}

\item \label{enum:cell1} $\cset$ is a set 
and $\cmset=\{\cmset(\lambda)\mid\lambda\in\cset\}$ 
a collection of finite sets such that 
\[
\cbasis(\placeholder,\placeholder)\colon 
{\textstyle\coprod_{\lambda\in\cset}}\,
\cmset(\lambda)\times\cmset(\lambda)\rightarrow\calg
\]
is an injective map with image forming a basis of $\calg$.  
For $S,T\in\cmset(\lambda)$ we write 
$\cbasis(S,T)=\cbas{S,T}{\lambda}$ from now on.

\renewcommand{\theenumi}{(\ref{definition:cell-algebra}.b)}
\renewcommand{\labelenumi}{\theenumi}

\item \label{enum:cell2} ${}^{\invo}$ is an anti-involution 
${}^{\invo}\colon\calg\to\calg$ such that 
$(\cbas{S,T}{\lambda})^\invo=\cbas{T,S}{\lambda}$.

\renewcommand{\theenumi}{(\ref{definition:cell-algebra}.c)}
\renewcommand{\labelenumi}{\theenumi}

\item \label{enum:cell3} 
$\ciset$ is a set of pairwise orthogonal, non-zero
idempotents, all fixed by 
${}^{\invo}$, i.e. $\ceps^{\invo}=\ceps$ for all $\ceps\in\ciset$.
Further, $\coset=\{\ord{\ceps}\mid\ceps\in\ciset\}$ 
is a set
of partial orders $\ord{\ceps}$ on $\cset$, and  
$\epsilon$ is a map 
$\epsilon\colon\coprod_{\lambda\in\cset}
\cmset(\lambda)\rightarrow\ciset$ sending $S$ 
to $\epsilon(S)=\ceps_{S}$ such that
\\
\noindent\begin{tabularx}{0.9\textwidth}{XX}
\begin{equation}\hspace{-7.5cm}\label{eq:idem-props-1}
\ceps\calg\ceps\,\cbas{S,T}{\lambda} \in \calg(\Ord{\ceps}\!\lambda),
\phantom{\begin{cases}
a,
\\
b,
\end{cases}}\hspace*{-1cm}
\end{equation} &
\begin{equation}\hspace{-6.5cm}\label{eq:idem-props-2}
\ceps\cbas{S,T}{\lambda}=
\begin{cases}
\cbas{S,T}{\lambda}, &\text{if }\ceps_{S}=\ceps,
\\
0, &\text{if }\ceps_{S}\neq\ceps,
\end{cases}
\end{equation}
\end{tabularx}\\
for all $\lambda\in\cset$, $S,T\in\cmset(\lambda)$ and $\ceps \in \ciset$.

\vspace*{.15cm}

\renewcommand{\theenumi}{(\ref{definition:cell-algebra}.d)}
\renewcommand{\labelenumi}{\theenumi}

\item \label{enum:cell4} For $\lambda\in\cset$, 
$S,T\in\cmset(\lambda)$ and $a\in\calg$ we have
\begin{gather*}
a \cbas{S,T}{\lambda} 
\in
{\textstyle\sum_{S^{\prime} \in \cmset(\lambda)}}\, 
r_{a}(S^{\prime},S)\,\cbas{S^{\prime},T}{\lambda} 
+ 
\calg(\ord{\ceps_{T}}\!\lambda)\ceps_{T},
\end{gather*}
with scalars $r_{a}(S^{\prime},S)\in\K$ only depending on 
$a,S,S^{\prime}$.
\end{enumerate}
\smallskip

We call the set 
$\{\cbas{S,T}{\lambda}\mid\lambda\in\cset,S,T\in \cmset(\lambda)\}$ 
a \textit{relative cellular basis}.

In the case $\ciset=\{1\}$ we call $\calg$ a 
\textit{cellular algebra}, and we write $\ord{}=\ord{1}$.

The whole setup is called \textit{graded} if the very 
same conditions as in \cite[Definition 2.1]{hm-klr-basis} 
are satisfied (which can be 
easily adapted to the relative case).
\end{definition}

Note that the notion of a cellular algebra 
in the sense of \cite[Definition 1.1]{gl-cellular} 
coincides with our definition here as one can easily check. 
In particular, a cellular algebra
is relative cellular, but not 
conversely as will follow e.g. from \fullref{example:cell-atypeA} combined with 
\fullref{lemma:bipartite}.

\subsubsection{The crucial examples}\label{subsubsec:rel-cell-examples}

The following examples partially appeared in \cite[Section 2E]{et-relcell}.

\begin{example}\label{example:cell-typeA}
Let $\Gg$ be a type $\typeA[3]$ graph. 
Then $\zigzag$ is cellular 
and its cell datum is as follows. The anti-involution ${}^{\invo}$ 
is the linear extension of 
the assignment which swaps source and target of paths.
Further, let $\cset=\{\ii[0]\ord{}\ii[1]\ord{}\ii[2]\ord{}\ii[3]\}$, 
with $\ii[0]$ playing the role of a dummy, and let
\[
\cmset(\ii[0])=\{\pathx{1}{2}\},
\quad
\cmset(\ii[1])=\{\ii[1],\pathx{2}{1}\},
\quad
\cmset(\ii[2])=\{\ii[2],\pathx{3}{2}\},
\quad
\cmset(\ii[3])=\{\ii[3]\},
\]
which determines the cells $\cmset(\ii)\times\cmset(\ii)$ 
by $\cbas{S,T}{\ii}=S\circ T^{\invo}$ for $(S,T)\in\cmset(\ii)\times\cmset(\ii)$. 
Clearly, this is 
a graded structure. Imposing the vertex condition 
$\bvec=\{\ii[1]\}$ gives a cellular structure for $\bzigzag$, but without 
any dummy.
\end{example}

\begin{example}\label{example:cell-atypeA}
Let $\Gg$ be a type $\atypeA[2]$ graph. 
Then $\zigzag$ is relative cellular 
and its cell datum is almost the same as in \fullref{example:cell-typeA}. 
The crucial differences are that there is no dummy cell, and 
one has two idempotents $\ciset=\{\ii[0],\ceps=\ii[1]+\ii[2]\}$ 
where the orderings are the same 
as in the finite case for $\ceps$, and the opposite for $\ii[0]$.
\end{example}

\subsection{The case of \texorpdfstring{$\zigzag$}{Z}}\label{subsec:rel-cell-zigzag}

\subsubsection{Construction}\label{subsubsec:cell-construction}

For any $\Gg$, let us define an anti-involution ${}^{\invo}\colon\zigzag\to\zigzag$ by
\begin{gather*}
\ii^{\invo}=\ii,
\quad\quad
\pathx{i}{j}^{\invo}=\pathx{j}{i},
\text{ for }\upathx{i}{j},
\end{gather*}
which determines ${}^{\invo}$ completely. 
Further, we will always use the rule
\[
\cbas{S,T}{\ii}=S\circ T^{\invo},\;\text{for }
(S,T)\in\cmset(\ii)\times\cmset(\ii),
\]
to give the cells $\cmset(\ii)\times\cmset(\ii)$ which 
we are going to define now in the cases where $\Gg$ is 
a finite or affine type $\typeA[]$ graph.

For type $\typeA[n]$ let 
the indexing set be 
$\cset=\{\ii[0]\ord{}\ii[1]\ord{}\dots\ord{}\ii[n]\}$ with the index $\ii[0]$ 
playing the role of a dummy. The cell sets $\cmset(\ii)$ are
\[
\cmset(\ii[0])=\{\pathx{1}{2}\},
\quad\quad
\cmset(\ii)=\{\ii,\pathx{i{+}1}{i}\},\text{ for }\ii\notin\{\ii[0],\ii[n]\},
\quad\quad
\cmset(\ii[n])=\{\ii[n]\}.
\]
The path length grading gives a way to view this datum as 
a graded datum by assigning to each element of $\cmset(\ii)$ 
the corresponding path length degree. 

For type $\atypeA[n]$ the cell datum is almost the 
same. The crucial differences are that the cells are now all of size four, and 
one has two idempotents $\ciset=\{\ceps=\ii[1]+\dots+\ii[n{-}1],\ii[n]\}$ 
where the orderings are the same 
as in the finite case for $\ceps$, and the opposite for $\ii[n]$. 

\begin{proposition}\label{proposition:cell-construction}
The above defines the structure of 
a graded relative cellular algebra on $\zigzag$. 
This structure is the structure of a 
graded cellular 
algebra in case $\Gg$ is of type $\typeA[n]$.  
\end{proposition}

\begin{proof}
We only prove the claim in case $\Gg$ is
a type $\atypeA[n]$ graph, the rest follows then 
directly from the fact that idempotent 
truncations of relative cellular algebras 
are relative cellular, cf. \cite[Proposition 2.8]{et-relcell}, 
as well as \cite[Example 2.3]{et-relcell}.
First, \ref{enum:cell1} follows from \fullref{lemma:basics3}, 
while \ref{enum:cell2} is immediate. The only statement 
from \ref{enum:cell3} which needs to be checked is 
\eqref{eq:idem-props-1} which follows since e.g. 
$\ii[n]\zigzag\ii[n]=\K\{\ii[n],\volele_{\ii[n]}\}$. The 
last condition \ref{enum:cell4} is verified via a small calculation.
The statement about the grading is immediate.
\end{proof}

\subsubsection{Elimination}\label{subsubsec:rel-cell-elim}

The crucial lemma, which is a consequence of 
\cite[Proposition 3.2]{kx-cellular} and 
\cite[Corollary 3.25]{et-relcell}, is:

\begin{lemmaqed}\label{lemma:the-cartan}
If $\zigzag$ is cellular, then 
its Cartan matrix $\cmatrix$ is positive definite. 
If $\zigzag$ is relative cellular, 
then $\cmatrix$ is positive semidefinite.
\end{lemmaqed}

For any relative cellular algebra,
recall that for each $\lambda\in\cset$ there is 
an associated \textit{cell module} $\ldmod[\lambda]$. 
Define the \textit{decomposition numbers} $d_{\mu\lambda}$ 
as the multiplicity of $\lsimple[\lambda]$ in $\ldmod[\mu]$, 
where $\lambda\in\cset_{0}$ correspond to idempotent cells. In particular, $d_{\mu\lambda}\in\N$.
In the case of $\zigzag$ 
we will identify $\cset_{0}=\{\ii\mid\ii\text{ vertex of }\Gg\}$, and
if we want to sort these numbers into a matrix $\dmatrix$, then 
we will always enumerate columns as indicated by the enumeration of the vertices. 
Further, we enumerate the rows by first using the elements from $\cset_{0}$ 
(in the enumeration of the vertices) followed by the remaining elements 
with a fixed, arbitrary enumeration.
By \cite[Proposition 3.6]{gl-cellular} or \cite[Theorem 3.23]{et-relcell} we get:

\begin{lemmaqed}\label{lemma:the-cartan-DD}
If $\zigzag$ is relative cellular, 
then $\cmatrix=\dmatrix^{\mathrm{T}}\dmatrix$ where $d_{\mu\ii}=0$ unless 
$\mu\ord{\ceps(\ii)}\ii$, and $d_{\ii\ii}=1$, where $\ceps(\ii)$ is 
the unique idempotent in $\ciset$ with $\ceps(\ii)\ii=\ii$.
\end{lemmaqed}

The following lemma and the 
ideas in its proof, 
will reappear throughout.

\begin{lemma}\label{lemma:bipartite}
If $\Gg$ is a bipartite graph which is not 
an $\typeADE$ graph, then $\zigzag$ 
is not cellular. If $\Gg$ is a bipartite graph which is not 
a finite or affine $\typeADE$ graph, then $\zigzag$ 
is not relative cellular.
\end{lemma}

\begin{proof}
Let $\Gg$ be bipartite. Then we claim the following, where we added the (graded) Cartan 
determinants for later use, since they can be computed using the numerical data 
given in the references below. 
\medskip

\noindent \textit{Claim.} $\cmatrix$ is positive definite 
if and only if $\Gg$ is a type $\typeADE$ graph.
$\cmatrix$ is positive semidefinite if and 
only if $\Gg$ is a finite or affine type $\typeADE$ graph. 
Further, in these cases
\begin{gather}\label{eq:the-data}
\begin{gathered}
\renewcommand{\arraystretch}{1.5}
\begin{tabular}{c|c|c|c|c|c}
 & $\typeA$ & $\typeD$ & $\typeE[6]$ & $\typeE[7]$ & $\typeE[8]$ \\ 
\hline 
$\mathrm{det}$ & $n+1$ & $4$ & $3$ & $2$ & $1$ \\ 
\hline 
$\qpar\mathrm{det}$ & $\qnumber{(n+1)}$ & $(1+\qpar^{2n{-}2})\qnumber{2}$ 
& $\begin{gathered}\phantom{\int}\!\!(1+\qpar^{10})\qnumber{2}\\ -\qpar^{6}\qnumber{1}\end{gathered}$ 
& $\begin{gathered}(1+\qpar^{12})\qnumber{2}\\ -\qpar^{6}\qnumber{2}\end{gathered}$ 
& $\begin{gathered}(1+\qpar^{14})\qnumber{2}\\ -\qpar^{6}\qnumber{3}\end{gathered}$ \\ 
\end{tabular}
\\
\renewcommand{\arraystretch}{1.5}
\begin{tabular}{c|c|c|c|c}
$\atypeA$ & $\atypeD$ & $\atypeE[6]$ & $\atypeE[7]$ & $\atypeE[8]$ \\ 
\hline 
$4\text{ or }0$ & $0$ & $0$ & $0$ & $0$ \\ 
\hline 
$\begin{gathered}\qpar^{2n+2}+1\\ +(-1)^{n}2\qpar^{n+1}\end{gathered}$ 
& $\begin{gathered}\phantom{\int}\!\!(1+\qpar^{2n})\qnumber{2}\\ -(\qpar^{4}+\qpar^{2n-4})\qnumber{2}\end{gathered}$ 
& $\begin{gathered}(1+\qpar^{12})\qnumber{2}\\ -(\qpar^{6}+\qpar^{6})\qnumber{2}\end{gathered}$ 
& $\begin{gathered}(1+\qpar^{14})\qnumber{2}\\ -(\qpar^{6}+\qpar^{8})\qnumber{2}\end{gathered}$ 
& $\begin{gathered}(1+\qpar^{16})\qnumber{2}\\ -(\qpar^{6}+\qpar^{10})\qnumber{2}\end{gathered}$ \\ 
\end{tabular}
\end{gathered}
\end{gather} 
are the Cartan determinants, respectively the graded 
Cartan determinants where we let $\qnumber{a}=1+\qpar^{2}+\dots+\qpar^{2a-2}+\qpar^{2a}$ 
for $a\in\Z_{\geq 0}$. 
\medskip

\noindent\textit{Proof of the claim.} 
The eigenvalues of the adjacency matrices of 
finite or affine type $\typeADE$ graphs are known, cf. \cite{sm-ADE} or 
\cite[Section 3.1.1]{bh-graphs}, 
and they all are in the interval $]-2,2[$ for the finite types, 
or in the interval $[-2,2]$ for the affine types. Moreover, by the same references,
the converse is also true: If $\Gg$ is a graph whose 
adjacency matrix has eigenvalues contained in $[-2,2]$, then $\Gg$ is a 
finite or affine type $\typeADE$ graph. In particular, all other graphs 
have a Perron--Frobenius eigenvalue strictly bigger than $2$. 
Finally, if $\Gg$ is bipartite, 
then they also have an eigenvalue strictly smaller than $-2$, and the 
claim follows by \fullref{proposition:cartan-zigzag}. 
\medskip

Then the lemma itself follows from this claim and \fullref{lemma:the-cartan}.
\end{proof}

\begin{lemma}\label{lemma:odd-circles}
If $\Gg$ is a type $\atypeA[]$ graph, then $\zigzag$ is not cellular.
\end{lemma}

\begin{proof}
Assume that $\zigzag$ is cellular.
Then, by \fullref{proposition:cartan-zigzag}
and \fullref{lemma:the-cartan-DD}, we get
\[
\scalebox{.8}{$\displaystyle
\left(
\begin{array}{cccccc}
2 & 1 & 0 & 0 & \cdots & 1
\\
1 & 2 & 1 & 0 & \cdots & 0
\\
0 & 1 & 2 & 1 & \ddots & \vdots
\\
0 & 0 & 1 & 2 & \ddots & \vdots
\\
\vdots & \ddots & \ddots & \ddots & \ddots & \vdots
\\
1 & 0 & 0 & 0 & 0 & 2
\end{array}
\right)
$}
=
\cmatrix=\dmatrix^{\mathrm{T}}\dmatrix
=
\scalebox{.8}{$\displaystyle
\left(
\begin{array}{cccc}
1 & d_{21} & d_{31} & \cdots
\\
d_{12} & 1 & d_{32} & \cdots
\\
\vdots & \ddots & \ddots & \ddots
\end{array}
\right)
$}
\scalebox{.8}{$\displaystyle
\left(
\begin{array}{ccc}
1 & d_{12} & \cdots
\\
d_{21} & 1 & \ddots
\\
d_{31} & d_{32} & \ddots
\\
\vdots & \vdots & \vdots
\end{array}
\right)
$}.
\]
We observe that this implies that each column of $\dmatrix$ contains precisely 
two non-zero entries, both of which are equal to $1$. One of these is the 
number $d_{\ii\ii}$, the other we will write as $d_{a(\ii),\ii}$. 
If $\ii,\ii[j]$ are non-adjacent vertices of $\Gg$, then it now follows that 
$a(\ii)\neq a(\ii[j])$ and $d_{\ii\ii[j]}=0=d_{\ii[j]\ii}$.  
It also follows that for $\upathx{i}{j}$ we can only have three 
distinct cases:
\[
\text{(i)}\colon
d_{\ii\ii[j]}=1,
\quad\quad
\text{(ii)}\colon
d_{\ii[j]\ii}=1,
\quad\quad
\text{(iii)}\colon
a(\ii)=a(\ii[j])\in\cset-\cset_{0}.
\]
The case (iii) does not give any solution as long as we have 
four or more vertices.
In the other cases the matrix $\dmatrix$ has to be of either form
\[
\text{(i) and (ii)}\colon
\dmatrix=
\scalebox{.8}{$\displaystyle
\left(
\begin{array}{cccc}
1 & 1 & 0 & 0
\\
0 & 1 & 1 & 0
\\
0 & 0 & 1 & 1
\\
1 & 0 & 0 & 1
\end{array}
\right)
$}
\;\;\text{or}\;\;
\dmatrix=
\scalebox{.8}{$\displaystyle
\left(
\begin{array}{cccc}
1 & 0 & 0 & 1
\\
1 & 1 & 0 & 0
\\
0 & 1 & 1 & 0
\\
0 & 0 & 1 & 1
\end{array}
\right)
$}
,\quad\quad
\text{(iii)}\colon
\dmatrix=
\scalebox{.8}{$\displaystyle
\left(
\begin{array}{ccc}
1 & 0 & 0
\\
0 & 1 & 0
\\
0 & 0 & 1
\\
\hline
1 & 1 & 1
\end{array}
\right)
$}
,
\]
here exemplified for four vertices for (i) and (ii). 
However, it is impossible to permute (by simultaneous row and column reenumeration) the two leftmost matrices
into upper triangular matrices, contradicting \fullref{lemma:the-cartan-DD}, 
and the claim follows for $n\geq 3$. 
For type $\atypeA[2]$ 
we need an extra argument to rule out (iii). In this 
case we would have four cell modules 
(corresponding to the rows of $\dmatrix$), three of which are simple 
and one of dimension three containing all simples in its filtration. 
Using \eqref{eq:loewy}, this gives a contradiction to \cite[Lemmas 2.9 and 2.10]{gl-cellular}, 
since there is no way to filter the projectives by these cell modules 
in any order compatible way, because all three indecomposable projective 
modules have non-equivalent socles and thus, could not 
agree with the socle of the three-dimensional cell module.
\end{proof}

\begin{lemma}\label{lemma:non-cell}
If $\Gg$ is not a bipartite graph, then $\zigzag$ is not cellular.
\end{lemma}

\begin{proof}
A non-bipartite graph has a subgraph of type $\atypeA[n{=}2m]$, 
and the claim follows by using \fullref{lemma:odd-circles} since idempotent truncations 
of cellular algebras are cellular by the same arguments as in
e.g. \cite[Section 6]{kx-cellular-inflations}.
\end{proof}

\begin{lemma}\label{lemma:cell-d4}
If $\Gg$ is a type $\typeD[4]$ graph, then $\zigzag$ is not relative cellular.
\end{lemma}

\begin{proof}
Assume relative cellularity. Then we get
\[
\scalebox{.8}{$\displaystyle
\left(
\begin{array}{cccc}
2 & 1 & 0 & 0
\\
1 & 2 & 1 & 1
\\
0 & 1 & 2 & 0
\\
0 & 1 & 0 & 2
\end{array}
\right)
$}
=
\cmatrix=\dmatrix^{\mathrm{T}}\dmatrix
=
\scalebox{.8}{$\displaystyle
\left(
\begin{array}{cccc|cc}
1 & d_{21} & d_{31} & d_{41} & d_{51} & \cdots
\\
d_{12} & 1 & d_{32} & d_{42} & d_{52} & \cdots
\\
d_{13} & d_{23} & 1 & d_{43} & d_{53} & \cdots
\\
d_{14} & d_{24} & d_{34} & 1 & d_{54} & \cdots
\end{array}
\right)
$}
\scalebox{.8}{$\displaystyle
\left(
\begin{array}{cccc}
1 & d_{12} & d_{13} & d_{14}
\\
d_{21} & 1 & d_{23} & d_{24}
\\
d_{31} & d_{32} & 1 & d_{34}
\\
d_{41} & d_{42} & d_{43} & 1
\\
\hline
d_{51} & d_{52} & d_{53} & d_{54}
\\
\vdots & \vdots & \vdots & \vdots
\end{array}
\right)
$}
\]
by the same arguments as before in the proof of 
\fullref{lemma:odd-circles}. Actually, we get the very same 
conditions for the entries of $\dmatrix$, depending on the connectivity 
of the vertices only. But this case is easier than the case of type 
$\atypeA[]$ since there is no possible solution.
\end{proof}

\begin{lemma}\label{lemma:non-cell-DE}
If $\Gg$ is a finite or affine type $\typeDE$ graph, 
then $\zigzag$ is not relative cellular. 
If $\Gg$ is a non-bipartite graph which is not 
a type $\atypeA[n{=}2m]$ graph, then $\zigzag$ is not relative cellular. 
\end{lemma}

\begin{proof}
Since idempotent truncations of relative cellular algebras would be 
relative cellular, cf. \cite[Proposition 2.8]{et-relcell}, the statement 
follows from \fullref{lemma:cell-d4}, and by observing that the 
three cases of a trivalent vertex containing no subgraph of type $\typeD[4]$, namely
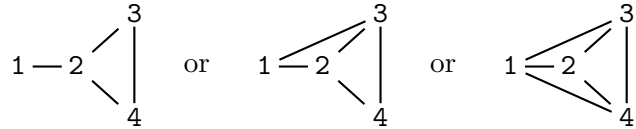
\[
\,
\xy
(0,0)*{
\begin{tikzcd}[ampersand replacement=\&,row sep=tiny,column sep=tiny,arrows={shorten >=-.5ex,shorten <=-.5ex},labels={inner sep=.05ex}]
\phantom{.} \& \phantom{.} \& \ii[3]\arrow[thick,-]{dd}{}
\\
\ii[1]\arrow[thick,-]{r}{} \& \ii[2]\arrow[thick,-]{ru}{}\arrow[thick,-]{rd}{} \& \phantom{.}
\\
\phantom{.} \& \phantom{.} \& \ii[4]
\end{tikzcd}};
\endxy
\quad\text{or}\quad
\,
\xy
(0,0)*{
\begin{tikzcd}[ampersand replacement=\&,row sep=tiny,column sep=tiny,arrows={shorten >=-.5ex,shorten <=-.5ex},labels={inner sep=.05ex}]
\phantom{.} \& \phantom{.} \& \ii[3]\arrow[thick,-]{dd}{}
\\
\ii[1]\arrow[thick,-]{r}{}\arrow[thick,-]{rru}{} \& \ii[2]\arrow[thick,-]{ru}{}\arrow[thick,-]{rd}{} \& \phantom{.}
\\
\phantom{.} \& \phantom{.} \& \ii[4]
\end{tikzcd}};
\endxy
\quad\text{or}\quad
\,
\xy
(0,0)*{
\begin{tikzcd}[ampersand replacement=\&,row sep=tiny,column sep=tiny,arrows={shorten >=-.5ex,shorten <=-.5ex},labels={inner sep=.05ex}]
\phantom{.} \& \phantom{.} \& \ii[3]\arrow[thick,-]{dd}{}
\\
\ii[1]\arrow[thick,-]{r}{}\arrow[thick,-]{rru}{}\arrow[thick,-]{rrd}{} \& \ii[2]\arrow[thick,-]{ru}{}\arrow[thick,-]{rd}{} \& \phantom{.}
\\
\phantom{.} \& \phantom{.} \& \ii[4]
\end{tikzcd}};
\endxy
\]
have $-1$ as an eigenvalue for their adjacency matrices, contradicting 
\fullref{lemma:the-cartan}.
\end{proof}

\subsubsection{The proof}\label{subsubsec:rel-cell-proof}

With the work already done we get:

\begin{proof}[Proof of \fullref{theorem:cell}]
\fullref{proposition:cell-construction} shows the 
existence of a cellular structure for finite or affine type $\typeA[]$ graphs. 
Conversely,\makeautorefname{lemma}{Lemmas} \fullref{lemma:bipartite}, \ref{lemma:non-cell} 
and \ref{lemma:non-cell-DE} prove\makeautorefname{lemma}{Lemmas} 
that none of the remaining cases 
can be (relative) cellular.
\end{proof}

\subsection{The case of \texorpdfstring{$\bzigzag$}{Z}}\label{subsec:rel-cell-bzigzag}

\subsubsection{Construction}\label{subsubsec:cell-construction-b}

Let $\Gg$ be a type $\typeA[n]$ graph, and let 
$\bzigzag(\typeA[n])$ be the associated zigzag algebra 
with vertex condition $\bvec=\{\ii[1]\}$ 
or $\bvec=\{\ii[n]\}$. In this case the construction of 
the cell datum works verbatim as for $\zigzag(\typeA[n])$ 
with the only difference that we do not need a dummy cell.
In particular, the following can be proven, mutatis mutandis, 
as \fullref{proposition:cell-construction}.

\begin{propositionqed}\label{proposition:cell-construction-b}
The above defines a graded cell datum for $\bzigzag(\typeA[n])$.  
\end{propositionqed}

\subsubsection{Elimination}\label{subsubsec:rel-cell-elim-b}

For completeness:

\begin{lemmaqed}\label{lemma:the-cartan-b} 
If $\bzigzag$ is relative cellular, 
then $\mathrm{det}(\cmatrix^{\bvec})\geq 0$.
\end{lemmaqed}

\begin{lemma}\label{lemma:cell-small-b}
Let $\Gg$ be of finite or affine $\typeADE$ graph. If $\bzigzag$ 
is relative cellular, then $\Gg$ is a type $\typeA[n]$ graph
and $\bvec=\{\ii[1]\}$ 
or $\bvec=\{\ii[n]\}$.
\end{lemma}

\makeautorefname{lemmaqed}{Lemmas}

\begin{proof}
The claim follows from a direct 
application of \fullref{lemma:the-cartan-b} 
and \ref{lemma:the-cartan-cell2}:\makeautorefname{lemmaqed}{Lemma}
For $\bvec=\{\ii[c]\}$ 
a small case-by-case check 
verifies that 
\begin{gather*}
\begin{tabular}{c|c|c|c|c|c|c|c|c|c}
type & $\typeA$ & $\typeD[],c{=}1$ & $\typeD[4],c{=}2$ & $\typeD[{>}4],c{>}1$ & $\typeE[]$ & $\atypeA[2m]$ & $\atypeA[2m+1]$ 
& $\atypeD[]$ & $\atypeE[]$
\\ 
\hline 
$\bvec$det & $n{+}1{-}c(n{-}c{+}1)$ & $0$ & ${<}0$ & ${<}0$ & ${<}0$ & $4{-}2m{-}1$ & ${<}0$ & ${<}0$ & ${<}0$
\\ 
\end{tabular}
\end{gather*} 
where we have used the data collected 
in \eqref{eq:the-data}.  
Since $n+1-c(n-c+1)\geq 0$ holds only 
in the cases $\bvec=\{\ii[1]\}$ or $\bvec=\{\ii[n]\}$, with one exception 
$\bvec=\{\ii[2]\}$ and $n=3$, and $4-2m-1\geq 0$ holds only if $m=1$, 
it remains to rule out those cases.
\medskip

\noindent\textit{Type $\typeA[3]$ with $\bvec=\{\ii[2]\}$.} 
Assuming relative cellularity, reciprocity would give
\[
\scalebox{.8}{$\displaystyle
\left(
\begin{array}{ccc}
2 & 1 & 0
\\
1 & 1 & 1
\\
0 & 1 & 2
\end{array}
\right)
$}
=
\cmatrix^{\bvec}=\dmatrix^{\mathrm{T}}\dmatrix
=
\scalebox{.8}{$\displaystyle
\left(
\begin{array}{ccc|cc}
1 & d_{21} & d_{31} & d_{41} & \cdots
\\
d_{21} & 1 & d_{32} & d_{42} & \cdots
\\
d_{31} & d_{32} & 1 & d_{43} & \cdots
\end{array}
\right)
$}
\scalebox{.8}{$\displaystyle
\left(
\begin{array}{ccc}
1 & d_{12} & d_{13}
\\
d_{21} & 1 & d_{23}
\\
d_{31} & d_{32} & 1
\\
\hline
d_{41} & d_{42} & d_{43}
\\
\vdots & \vdots & \vdots
\end{array}
\right)
$},
\]
which is impossible as one can easily check.
\medskip

\noindent\textit{Type $\typeD$ with $\bvec=\{\ii[1]\}$.} First assume that $n\geq 5$. Then 
there would be a type $\typeD[4]$ subgraph without vertex 
condition and idempotent truncation gives a contradiction to 
\fullref{lemma:cell-d4}. The remaining 
case, $\typeD[4]$ with vertex on any leaf, can be 
ruled out with the same matrix comparison as we have used above.
\medskip

\noindent\textit{Type $\atypeA[2]$.}
We first observe that in this case
all three possibilities to impose the vertex condition 
give isomorphic algebras. So let $\bvec=\{\ii[0]\}$. 
Then reciprocity can only work with
\[
\scalebox{.8}{$\displaystyle
\left(
\begin{array}{ccc}
1 & 1 & 1
\\
1 & 2 & 1
\\
1 & 1 & 2
\end{array}
\right)
$}
=
\cmatrix^{\bvec}=\dmatrix^{\mathrm{T}}\dmatrix
=
\scalebox{.8}{$\displaystyle
\left(
\begin{array}{ccc}
1 & 0 & 0
\\
1 & 1 & 0
\\
1 & 0 & 1
\end{array}
\right)
$}
\scalebox{.8}{$\displaystyle
\left(
\begin{array}{ccc}
1 & 1 & 1
\\
0 & 1 & 0
\\
0 & 0 & 1
\end{array}
\right)
$}
\]
as the usual reasoning shows. This would mean that we have 
three cell modules, with $\ldmod[0]$ being of dimension three, 
and the other two being simple. By \eqref{eq:loewy}, we also 
have projectives
\[
\lprojb[0]
=
\begin{gathered}
\ii[0]
\\
\pathx{\ii[1]}{\ii[0]}
\;
\pathx{\ii[2]}{\ii[0]}
\\
\phantom{\volele_{2}}
\end{gathered}
,\quad\quad
\lproj[1]
=
\begin{gathered}
\ii[1]
\\
\pathx{\ii[0]}{\ii[1]}
\;
\pathx{\ii[2]}{\ii[1]}
\\
\volele_{1}
\end{gathered}
,\quad\quad
\lproj[2]
=
\begin{gathered}
\ii[2]
\\
\pathx{\ii[0]}{\ii[2]}
\;
\pathx{\ii[1]}{\ii[2]}
\\
\volele_{2}
\end{gathered}
\]
with head $\lprojb[0]$ being $\lsimple[0]$,
and the heads of $\lproj[1]$ and $\lproj[2]$ 
are $\lsimple[1]=\ldmod[1]$ and $\lsimple[2]=\ldmod[2]$, 
respectively. However, any indecomposable 
projective has a filtration by cell modules, 
see \cite[Proposition 3.19]{et-relcell}, 
but only $\ldmod[0]$ contains $\lsimple[0]$, which gives a contradiction.
\medskip

If $\bvec$ has more than one vertex, then, 
with the same idempotent truncation arguments as above,
we only need to rule out the case of $\Gg$ being a type 
$\typeA[2]$ graph with $\bvec=\{\ii[1],\ii[2]\}$ or a type $\typeD[4]$ graph
for any $\bvec$, or a type $\typeA[3]$ graph with $\bvec=\{\ii[2]\}$.
Most of these have already been ruled out above.
By symmetry, for $\typeD[4]$ we just calculate
\begin{gather*}
\begin{tabular}{c|c|c|c|c}
type & $\typeD[4],\bvec=\{\ii[1],\ii[2]\}$ &
$\typeD[4],\bvec=\{\ii[1],\ii[3]\}$ & 
$\typeD[4],\bvec=\{\ii[1],\ii[2],\ii[3]\}$ &
$\typeD[4],\bvec=\{\ii[1],\ii[2],\ii[3],\ii[4]\}$
\\ 
\hline 
$\bvec$-det & 
$-4$ & $-1$ & $-3$ & $-2$
\\ 
\end{tabular}
\end{gather*}
which rules out this case by \fullref{lemma:the-cartan-b}. 
Finally, the case of a type $\typeA[2]$ 
graph with $\bvec=\{\ii[1],\ii[2]\}$ 
one again gets no matrix solution for reciprocity.
\end{proof}

\begin{lemma}\label{lemma:cell-big-b}
If $\Gg$ is not a finite or affine $\typeADE$ graph, then 
$\bzigzag$ is not relative cellular.
\end{lemma}

\begin{proof}
Assume that $\bzigzag$ is relative cellular.
By the usual subgraph-truncation argument and by 
\fullref{lemma:cell-d4} as well as \fullref{lemma:cell-small-b}, 
we see that $\Gg$ can not contain a type $\typeD[4]$ subgraph 
(with or without any $\ii[c]\in\bvec$ being on this subgraph). 
Using the same arguments and \fullref{lemma:cell-small-b}, we 
see that $\Gg$ can not contain 
a type $\atypeA[]$ subgraph with $\ii[c]\in\bvec$ on it. Since 
$\Gg$ is not a type $\atypeA[]$ graph itself, it remains to rule 
out one case, i.e.
\[
\,
\xy
(0,0)*{
\begin{tikzcd}[ampersand replacement=\&,row sep=tiny,column sep=tiny,arrows={shorten >=-.5ex,shorten <=-.5ex},labels={inner sep=.05ex}]
\phantom{.} \& \phantom{.} \& \ii[3]\arrow[thick,-]{dd}{}
\\
\ii[1]\arrow[thick,-]{r}{} \& \ii[2]\arrow[thick,-]{ru}{}\arrow[thick,-]{rd}{} \& \phantom{.}
\\
\phantom{.} \& \phantom{.} \& \ii[4]
\end{tikzcd}};
\endxy
\quad
\text{for }\bvec=\{\ii[1]\}.
\]
(All other configurations of a trivalent vertex containing no 
$\typeD[4]$ subgraph contain a subgraph of type $\atypeA[2]$ 
with at least one vertex condition.)
The same reciprocity reasoning as before give only 
one potential solution, namely
\[
\scalebox{.8}{$\displaystyle
\left(
\begin{array}{cccc}
1 & 1 & 0 & 0
\\
1 & 2 & 1 & 1
\\
0 & 1 & 2 & 1
\\
0 & 1 & 1 & 2
\end{array}
\right)
$}
=
\cmatrix^{\bvec}=\dmatrix^{\mathrm{T}}\dmatrix
=
\scalebox{.8}{$\displaystyle
\left(
\begin{array}{cccc}
1 & 0 & 0 & 0
\\
1 & 1 & 0 & 0
\\
0 & 1 & 1 & 0
\\
0 & 1 & 0 & 1
\end{array}
\right)
$}
\scalebox{.8}{$\displaystyle
\left(
\begin{array}{cccc}
1 & 1 & 0 & 0
\\
0 & 1 & 1 & 1
\\
0 & 0 & 1 & 0
\\
0 & 0 & 0 & 1
\end{array}
\right)
$},
\]
from where we can read of the cell modules.
Again, this gives a contradiction to the 
filtration of the indecomposable projectives 
by cell modules.
\end{proof}

\subsubsection{The proof}\label{subsubsec:rel-cell-proof-b}

Nothing remains to be done:

\makeautorefname{lemma}{Lemmas}

\begin{proof}[Proof of \fullref{theorem:cellb}]
Combine \fullref{proposition:cell-construction-b} with
\fullref{lemma:cell-small-b} and \ref{lemma:cell-big-b}
\end{proof}

\makeautorefname{lemma}{Lemma}

In fact, by using \cite[Proposition 3.9]{at-tilting}, \cite[Theorem 3.9]{ast-cell} and 
\cite[Theorem A.13]{bt-cores}, a non-trivial 
consequence of \fullref{theorem:cellb} is that zigzag algebras 
(with or without vertex condition) can only appear 
as endomorphism algebras of duality stable tilting modules in certain highest weight categories 
if and only if $\Gg$ is a type $\typeA[]$ graph.
%
\section{Quasi-heredity}\label{section:qh}

\subsubsection{A brief reminder}\label{subsubsec:qh-reminder}

We start by recalling the definition of a quasi-hereditary algebra
as it appears e.g. in \cite[Below Example 3.3]{cps-qh}. To this end,  
a (two-sided) ideal $\qhideal$ in a finite-dimensional algebra $\qhalg$, which is 
projective as a left $\qhalg$-module, is called 
\textit{hereditary} if $\qhideal\qhideal=\qhideal$ and 
$\qhideal\algstuff{R}\mathrm{ad}(\qhalg)\qhideal=0$ hold, where 
$\algstuff{R}\mathrm{ad}(\qhalg)$ is the \textit{Jacobson radical} of $\qhalg$.

\begin{definition}\label{definition:qh}
A finite-dimensional algebra $\qhalg$ is called 
\textit{quasi-hereditary} if there exists a chain of ideals
\[
0=\qhideal_{0}\subset\qhideal_{1}\subset\dots\subset\qhideal_{k-1}\subset\qhideal_{k}=\qhalg,
\]
for some $k\in\Z_{\geq 1}$, such that the 
quotients $\qhideal_{l}/\qhideal_{l-1}$ 
are hereditary ideals in $\qhalg/\qhideal_{l-1}$.
\end{definition}

A chain as in \fullref{definition:qh} is called a \textit{hereditary chain}.

\subsubsection{The crucial example}\label{subsubsec:qh-examples}

\begin{example}\label{example:qh-typeA}
Let $\Gg$ be a type $\typeA[3]$ graph, and let $\bvec=\{\ii[1]\}$. 
Let $\qhideal_{1}=\bzigzag(\ii[1])\bzigzag$, 
$\qhideal_{2}=\bzigzag(\ii[1]+\ii[2])\bzigzag$ 
and $\qhideal_{3}=\bzigzag(\ii[1]+\ii[2]+\ii[3])\bzigzag$, i.e.
\begin{gather*}
\qhideal_{1}=\K\{
\ii[1],\pathx{2}{1},\pathx{1}{2},\volele_{\ii[2]}
\},
\quad
\qhideal_{2}=\K\{
\ii[2],\pathx{3}{2},\pathx{2}{3},\volele_{\ii[3]}
\}\oplus\qhideal_{1},
\quad
\qhideal_{3}=\K\{
\ii[3]
\}\oplus\qhideal_{1}\oplus\qhideal_{2},
\end{gather*}
where we recall that $\volele_{\ii[2]}=\pathxx{2}{1}{2}$ 
and $\volele_{\ii[3]}=\pathxx{3}{2}{3}$, while $\volele_{\ii[1]}=0$ 
because $\bvec=\{\ii[1]\}$. It follows that these form a hereditary chain.
\end{example}

\subsection{The case of \texorpdfstring{$\zigzag$}{Z}}\label{subsec:qh-zigzag}

\begin{proof}[Proof of \fullref{theorem:qh}]
The discussion about simple $\zigzag$-modules 
in \fullref{section:prelim} shows that the Jacobson 
radical $\algstuff{R}\mathrm{ad}(\zigzag)$ 
is equal to the span of all paths of positive length. 
In particular, all volume elements $\volele_{\ii}$ are in $\algstuff{R}\mathrm{ad}(\zigzag)$.
Hence, $\zigzag$ can never be 
quasi-hereditary since $\qhideal_{1}$ could not contain any 
idempotent, because any such idempotent would be a sum of vertex idempotents. 
Thus, we would have $\qhideal_{1}\algstuff{R}\mathrm{ad}(\zigzag)\qhideal_{1}\neq 0$.
\end{proof}

\subsection{The case of \texorpdfstring{$\bzigzag$}{Zb}}\label{subsec:qh-zigzag-b}

Before we start, note that the Jacobson 
radical $\algstuff{R}\mathrm{ad}(\bzigzag)$ of $\bzigzag$ is, as in the case of $\zigzag$, the 
span of all paths of positive length.

\subsubsection{Construction}\label{subsubsec:qh-construction}

Let $\Gg$ be a type $\typeA[n]$ graph, and 
let $\bzigzag$ the associated zigzag 
algebra with vertex condition $\bvec=\{\ii[1]\}$. 
Let $\qhideal_{0}=0$, and for each $\ii\in\{1,\dots,n\}$ 
we define
\begin{gather}\label{eq:qh-cells}
\qhideal_{i}=\bzigzag(\ii[1]+\dots+\ii[i])\bzigzag.
\end{gather}
Similarly, but reversing the summation order, in case $\bvec=\{\ii[n]\}$.

The following is just a summary of \fullref{example:qh-typeA}.

\begin{proposition}\label{proposition:qh-construction}
Assignment \eqref{eq:qh-cells} gives rise to a hereditary chain for $\bzigzag(\typeA[n])$ 
with vertex condition $\bvec=\{\ii[1]\}$. Similarly in case of $\bvec=\{\ii[n]\}$.
\end{proposition}

\begin{proof}
By construction, the $\qhideal_{i}$ are ideals in $\zigzag$ and form 
a chain as in \fullref{definition:qh}. Moreover, $\qhideal_{i}/\qhideal_{i{-}1}$ contains 
precisely one idempotent (which we identified with $\ii$), paths 
of length one either starting or ending at $\ii$ and the 
volume element $\volele_{\ii[i{+}1]}$. 
That is, we have
\begin{gather*}
\qhideal_{1}=\K\{
\ii[1],\pathx{2}{1},\pathx{1}{2},\volele_{\ii[2]}=\pathxx{2}{1}{2}
\},
\quad
\qhideal_{n}/\qhideal_{n{-}1}
=\K\{
\ii[n]
\},
\\
\qhideal_{i}/\qhideal_{i{-}1}
=
\K\{\ii,\pathx{i{+}1}{i},\pathx{i}{i{+}1},\volele_{\ii[i{+}1]}=\pathxx{i{+}1}{i}{i{+}1}\},
\quad i\in\{2,\dots,n-1\},
\end{gather*}
which shows that $\qhideal_{i}/\qhideal_{i{-}1}$ 
is an idempotent ideal. 
To check the other two conditions we observe that
$\bzigzag(\typeA[n])/\qhideal_{i}\cong\bzigzag(\typeA[n{-}i])$ 
with vertex condition imposed on its first leaf.
This means we only need to check these two conditions for $\qhideal_{1}$ 
where we get
\[
\qhideal_{1}\cong\lproj[1]\oplus\lproj[1],
\quad\quad
\qhideal_{1}\algstuff{R}\mathrm{ad}(\bzigzag)\qhideal_{1}=0,
\]
cf. \eqref{eq:loewy}. By symmetry, the same works with the vertex 
condition $\bvec=\{\ii[n]\}$.
\end{proof}

\subsubsection{Elimination}\label{subsubsec:qh-elim}

We recall a consequence of 
a classical result, cf. \cite[Proposition 1.3]{bf-qh}.

\begin{lemmaqed}\label{lemma:the-cartan-qh}
If $0=\qhideal_{0}\subset\qhideal_{1}\subset
\dots\subset\qhideal_{k-1}\subset\qhideal_{k}=\bzigzag$,
is a hereditary chain of $\bzigzag$, then 
$\mathrm{det}(\cmatrix^{\bvec}(\bzigzag/\qhideal_{l{-}1}))=1$ for all $l\in\{1,\dots,k\}$.
\end{lemmaqed}

\begin{lemma}\label{lemma:qh-small}
Let $\Gg$ be of finite or affine $\typeADE$ graph,
and assume that $\bvec=\{\ii[c]\}$. If $\bzigzag$ 
is quasi-hereditary, then $\Gg$ is a type $\typeA[n]$ graph
and $\bvec=\{\ii[1]\}$ or $\bvec=\{\ii[n]\}$.
\end{lemma}

\begin{proof}
By \fullref{lemma:the-cartan-qh} and
very similar to the proof of \fullref{lemma:cell-small-b} 
(by calculating determinants), 
it remains to check the case that $\Gg$ is a type $\atypeA[2]$ graph 
with one vertex condition.
To this end, we first observe that in this case
all three possibilities to impose the vertex condition 
give isomorphic algebras. So let $\bvec=\{\ii[0]\}$. 
Then, by e.g. \cite[Corollary 1.2]{bf-qh}, 
one would need to set
\[
\qhideal_{1}=\K\{
\ii[0],
\pathx{1}{0},\pathx{0}{1},\volele_{\ii[1]},
\pathx{2}{0},\pathx{0}{1},\volele_{\ii[2]}
\}.
\]
But we have $\mathrm{det}(\cmatrix^{\bvec}(\bzigzag/\qhideal_{1}))=0$, 
contradicting \fullref{lemma:the-cartan-qh}.
\end{proof}

\begin{lemma}\label{lemma:qh-big}
If $\Gg$ is not a type $\typeA[n]$ graph with $\bvec=\{\ii[1]\}$ or $\bvec=\{\ii[n]\}$, 
then $\bzigzag$ is not quasi-hereditary.
\end{lemma}

\begin{proof}
Assume that $\bzigzag$ is quasi-hereditary.
Again, by \cite[Corollary 1.2]{bf-qh}, 
the ideal $\qhideal_{1}$ has to contain a 
primitive idempotent for some $\ii[c]\in\bvec$, 
and we get
\begin{align*}
\qhideal_{1}
&=
\K\{
\ii[c],\pathx{j}{c},\pathx{c}{j},\volele_{\ii[j]},
\pathx{d}{c},\pathx{c}{d}\mid \upathx{c}{j},\ii[j]\notin\bvec,\upathx{c}{d},\ii[d]\in\bvec
\}
\\
&\cong
\lprojb[c]
\oplus
{\textstyle\bigoplus_{\upathx{c}{j},\ii[j]\notin\bvec}}\,
\K\{
\pathx{c}{j},\volele_{\ii[j]}
\}
\oplus
{\textstyle\bigoplus_{\upathx{c}{d},\ii[d]\in\bvec}}\,
\K\{
\pathx{c}{d}
\},
\end{align*}
where the decomposition follows from the structure 
of $\bzigzag$. If $\bvec$ does not contain a leaf, then 
each indecomposable projective of $\bzigzag$ is of dimension at least three, 
cf. \eqref{eq:loewy}, and $\qhideal_{1}$ can not be a projective 
$\bzigzag$-module. Thus, assume that $\ii[c]\in\bvec$ is a leaf 
that gives $\qhideal_{1}$. By the same argument as above, 
$\ii[c]$ could not have a neighbor $\ii[c]\in\bvec$ since 
each indecomposable projective of $\bzigzag$ in this case is of dimension at least two.
In the remaining case, i.e. $\ii[c]\in\bvec$ is a leaf and its 
neighbor $\upathx{c}{j}$ is not in $\bvec$, we can consider $\bzigzag/\qhideal_{1}$ 
which would be quasi-hereditary, cf. \cite[Below Example 3.3]{cps-qh}. But this 
would recursively give a contradiction, since we have 
$\bzigzag/\qhideal_{1}\cong\zigzag^{\bvec-\ii[c]}(\Gg-\ii[c])$ 
where the vertex condition for $\zigzag^{\bvec-\ii[c]}(\Gg-\ii[c])$ is at $\ii[j]$: 
If $|\bvec|\geq 2$, then, at one point, one has two neighboring 
vertex conditions, contradicting the above observation. If $\bvec=\{\ii[c]\}$, then 
one will have a finite or affine type $\typeADE$ graph 
with one vertex condition. However, by \fullref{lemma:qh-small}, 
the only case where this would 
not give a contradiction is the case with $\Gg$ being a type $\typeA[n]$ graph 
with $\bvec=\{\ii[1]\}$ or $\bvec=\{\ii[n]\}$.
\end{proof}

\subsubsection{The proof}\label{subsubsec:qh-proof}

We collect the harvest:

\begin{proof}[Proof of \fullref{theorem:qhb}]
Combine \fullref{proposition:qh-construction}, which constructs 
the hereditary chain, with \fullref{lemma:qh-small}, which rules out all other cases.
\end{proof}
%
\section{Koszulity}\label{section:koszul}

\subsubsection{A brief reminder}\label{subsubsec:koszul-reminder}

We start one definition of 
a Koszul algebra, cf. \cite[Definition 1.1.2]{bgs-koszul} 
or \cite[Section 2.2]{pp-quadratic}. 
Recall here that a \textit{linear projective 
resolution} of a graded module $\algstuff{M}$ of a finite-dimensional, 
positively graded algebra $\kalg$ 
is an exact sequence
\begin{gather}\label{eq:koszul-reso}
\begin{tikzcd}[ampersand replacement=\&,row sep=small,column sep=small,arrows={shorten >=-.5ex,shorten <=-.5ex},labels={inner sep=.05ex}]
\cdots
\arrow[->]{r}{}
\&
\qpar^{2}
\algstuff{Q}_{2}
\arrow[->]{r}{}
\&
\qpar
\algstuff{Q}_{1}
\arrow[->]{r}{}
\&
\algstuff{Q}_{0}
\arrow[->>]{r}{}
\&
\algstuff{M},
\end{tikzcd}
\end{gather}
with graded projective $\kalg$-modules $\qpar^{t}\algstuff{Q}_{t}$ 
(for us this is the $t^{\mathrm{th}}$ part of the resolution) 
generated in degree $t$, and $\kalg$-equivariant maps of degree $0$.
Using our grading conventions from \fullref{section:prelim}, this is the same 
data as giving an exact sequence of homogeneous, $\kalg$-equivariant maps of degree $1$ 
between the graded projectives $\algstuff{Q}_{t}$.

\begin{definition}\label{definition:koszul}
A finite-dimensional, positively graded algebra $\kalg$ is called 
\textit{Koszul} if its degree $0$ part is semisimple and each simple $\kalg$-module admits 
a linear projective resolution.
\end{definition}

For us only the property of having a linear projective resolution will 
play a role; our algebras are evidently finite-dimensional, positively graded 
and have a semisimple degree $0$ part. 
Moreover, up to shifts, we can focus on simples which are concentrated in degree $0$.

\subsubsection{The crucial examples}\label{subsubsec:koszul-examples}

\begin{example}\label{example-koszul-affineA}
Let $\Gg$ be a type $\atypeA[2]$ graph. In this case it is easy to write 
down a linear projective resolution of $\lsimple[0]$ 
(of course, the others are similar):
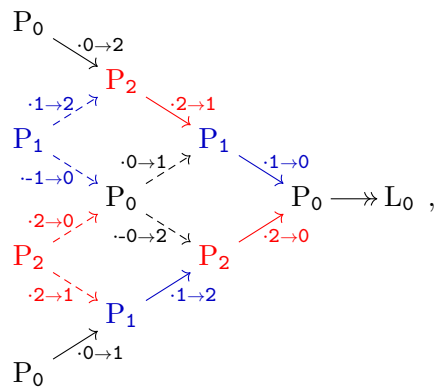
\[
\xy
(0,0)*{
\begin{tikzcd}[ampersand replacement=\&,row sep=tiny,column sep=small,arrows={shorten >=-.5ex,shorten <=-.5ex},labels={inner sep=.05ex}]
\lproj[0]\arrow[thin,->]{rd}{\cdot\obstuff{0}\scalebox{.8}{$\to$}\obstuff{2}} \& \phantom{.} \& \phantom{.} \& \phantom{.} \& \phantom{.}
\\
\phantom{.} \& {\color{myred}\lproj[2]}\arrow[thin,myred,->]{rd}{\cdot\obstuff{2}\scalebox{.8}{$\to$}\obstuff{1}} \& \phantom{.} \& \phantom{.} \& \phantom{.}
\\
{\color{myblue}\lproj[1]}\arrow[thin,densely dashed,myblue,->]{rd}[swap]{\cdot\text{-}\obstuff{1}\scalebox{.8}{$\to$}\obstuff{0}}\arrow[thin,densely dashed,myblue,->]{ru}{\cdot\obstuff{1}\scalebox{.8}{$\to$}\obstuff{2}} \& \phantom{.} \& {\color{myblue}\lproj[1]}\arrow[thin,myblue,->]{rd}{\cdot\obstuff{1}\scalebox{.8}{$\to$}\obstuff{0}} \& \phantom{.} \& \phantom{.}
\\
\phantom{.} \& \lproj[0]\arrow[thin,densely dashed,->]{rd}[swap]{\cdot\text{-}\obstuff{0}\scalebox{.8}{$\to$}\obstuff{2}}\arrow[thin,densely dashed,->]{ru}{\cdot\obstuff{0}\scalebox{.8}{$\to$}\obstuff{1}} \& \phantom{.} \& \lproj[0]\arrow[thin,->>]{r}{} \& \lsimple[0]
\\
{\color{myred}\lproj[2]}\arrow[thin,densely dashed,myred,->]{rd}[swap]{\cdot\obstuff{2}\scalebox{.8}{$\to$}\obstuff{1}}\arrow[thin,densely dashed,myred,->]{ru}{\cdot\obstuff{2}\scalebox{.8}{$\to$}\obstuff{0}} \& \phantom{.} \& {\color{myred}\lproj[2]}\arrow[thin,myred,->]{ru}[swap]{\cdot\obstuff{2}\scalebox{.8}{$\to$}\obstuff{0}} \& \phantom{.} \& \phantom{.}
\\
\phantom{.} \& {\color{myblue}\lproj[1]}\arrow[thin,myblue,->]{ru}[swap]{\cdot\obstuff{1}\scalebox{.8}{$\to$}\obstuff{2}} \& \phantom{.} \& \phantom{.} \& \phantom{.}
\\
\lproj[0]\arrow[thin,->]{ru}[swap]{\cdot\obstuff{0}\scalebox{.8}{$\to$}\obstuff{1}} \& \phantom{.} \& \phantom{.} \& \phantom{.} \& \phantom{.}
\end{tikzcd}};
\endxy,
\]
where the rightmost map is the projection, while the other maps 
are given by post-composition (which commutes with the left action 
given by pre-composition) with the corresponding paths, where the 
dashed arrows hit linear combinations in the kernels. For example, 
because $\pathxx{0}{1}{0}=\volele_{\ii[0]}=\pathxx{0}{2}{0}$, the element 
$\pathx{0}{1}-\pathx{0}{2}$ is in the $1^{\mathrm{st}}$ kernel and 
we use two maps from $\lproj[0]$ to compensate for it. Since 
these are of degree $1$, we get a linear projective resolution by 
continuation of the sketched pattern.
\end{example}

\begin{example}\label{example-koszul-finiteD}
Let $\Gg$ be a type $\typeD[4]$ graph. We try to 
resolve ${\color{myblue}\lsimple[1]}$ and ${\color{myred}\lsimple[2]}$ linearly 
by projectives:
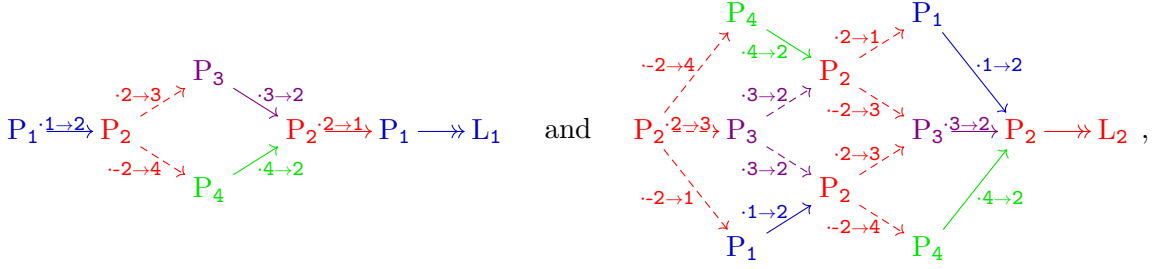
\[
\xy
(0,0)*{
\begin{tikzcd}[ampersand replacement=\&,row sep=tiny,column sep=small,arrows={shorten >=-.5ex,shorten <=-.5ex},labels={inner sep=.05ex}]
\phantom{.} \& \phantom{.} \& {\color{mypurple}\lproj[3]}\arrow[thin,mypurple,->]{rd}{\cdot\obstuff{3}\scalebox{.8}{$\to$}\obstuff{2}} \& \phantom{.} \& \phantom{.} \& \phantom{.}
\\
{\color{myblue}\lproj[1]}\arrow[thin,myblue,->]{r}[xshift=-.11cm]{\cdot\obstuff{1}\scalebox{.8}{$\to$}\obstuff{2}} \& {\color{myred}\lproj[2]}\arrow[thin,densely dashed,myred,->]{rd}[swap]{\cdot\text{-}\obstuff{2}\scalebox{.8}{$\to$}\obstuff{4}}\arrow[thin,densely dashed,myred,->]{ru}{\cdot\obstuff{2}\scalebox{.8}{$\to$}\obstuff{3}} \& \phantom{.} \& {\color{myred}\lproj[2]}\arrow[thin,myred,->]{r}[xshift=-.11cm]{\cdot\obstuff{2}\scalebox{.8}{$\to$}\obstuff{1}} \& {\color{myblue}\lproj[1]}\arrow[thin,myblue,->>]{r}{} \& {\color{myblue}\lsimple[1]}
\\
\phantom{.} \& \phantom{.} \& {\color{mygreen}\lproj[4]}\arrow[thin,mygreen,->]{ru}[swap]{\cdot\obstuff{4}\scalebox{.8}{$\to$}\obstuff{2}} \& \phantom{.} \& \phantom{.} \& \phantom{.}
\end{tikzcd}};
\endxy
\quad
\text{and}
\quad
\xy
(0,0)*{
\begin{tikzcd}[ampersand replacement=\&,row sep=tiny,column sep=small,arrows={shorten >=-.5ex,shorten <=-.5ex},labels={inner sep=.05ex}]
\phantom{.} \& {\color{mygreen}\lproj[4]}\arrow[thin,mygreen,->]{rd}[swap]{\cdot\obstuff{4}\scalebox{.8}{$\to$}\obstuff{2}} \& \phantom{.} \& {\color{myblue}\lproj[1]}\arrow[thin,myblue,->]{rdd}{\cdot\obstuff{1}\scalebox{.8}{$\to$}\obstuff{2}} \& \phantom{.} \& \phantom{.}
\\
\phantom{.} \& \phantom{.} \& {\color{myred}\lproj[2]}\arrow[thin,densely dashed,myred,->]{rd}[swap]{\cdot\text{-}\obstuff{2}\scalebox{.8}{$\to$}\obstuff{3}}\arrow[thin,densely dashed,myred,->]{ru}{\cdot\obstuff{2}\scalebox{.8}{$\to$}\obstuff{1}} \& \phantom{.} \& \phantom{.} \& \phantom{.}
\\
{\color{myred}\lproj[2]}\arrow[thin,densely dashed,myred,->]{ruu}{\cdot\text{-}\obstuff{2}\scalebox{.8}{$\to$}\obstuff{4}}\arrow[thin,densely dashed,myred,->]{rdd}[swap]{\cdot\text{-}\obstuff{2}\scalebox{.8}{$\to$}\obstuff{1}}\arrow[thin,densely dashed,myred,->]{r}[xshift=-.11cm]{\cdot\obstuff{2}\scalebox{.8}{$\to$}\obstuff{3}} \& {\color{mypurple}\lproj[3]}\arrow[thin,densely dashed,mypurple,->]{rd}[swap]{\cdot\obstuff{3}\scalebox{.8}{$\to$}\obstuff{2}}\arrow[thin,densely dashed,mypurple,->]{ru}{\cdot\obstuff{3}\scalebox{.8}{$\to$}\obstuff{2}} \& \phantom{.} \& {\color{mypurple}\lproj[3]}\arrow[thin,mypurple,->]{r}[xshift=-.11cm]{\cdot\obstuff{3}\scalebox{.8}{$\to$}\obstuff{2}} \& {\color{myred}\lproj[2]}\arrow[thin,myred,->>]{r}{} \& {\color{myred}\lsimple[2]}
\\
\phantom{.} \& \phantom{.} \& {\color{myred}\lproj[2]}\arrow[thin,densely dashed,myred,->]{rd}[swap]{\cdot\text{-}\obstuff{2}\scalebox{.8}{$\to$}\obstuff{4}}\arrow[thin,densely dashed,myred,->]{ru}{\cdot\obstuff{2}\scalebox{.8}{$\to$}\obstuff{3}} \& \phantom{.} \& \phantom{.} \& \phantom{.}
\\
\phantom{.} \& {\color{myblue}\lproj[1]}\arrow[thin,myblue,->]{ru}{\cdot\obstuff{1}\scalebox{.8}{$\to$}\obstuff{2}} \& \phantom{.} \& {\color{mygreen}\lproj[4]}\arrow[thin,mygreen,->]{ruu}[swap]{\cdot\obstuff{4}\scalebox{.8}{$\to$}\obstuff{2}} \& \phantom{.} \& \phantom{.}
\end{tikzcd}};
\endxy,
\]
where we use the same conventions as in \fullref{example-koszul-affineA}.
The $4^{\mathrm{th}}$ parts have kernels supported in degree $2$, and we are stuck. 
(Note that this happens after the same number of steps.) 
As we will see, the same holds 
for all type $\typeADE$ graphs.

In contrast, if we add the vertex condition $\bvec=\{\ii[1]\}$ to 
the type $\typeA[]$ graph, then the linear projective resolutions 
exist and are all finite. This happens since ${\color{myblue}\lprojb[1]}$ 
will not contain a volume element, e.g.
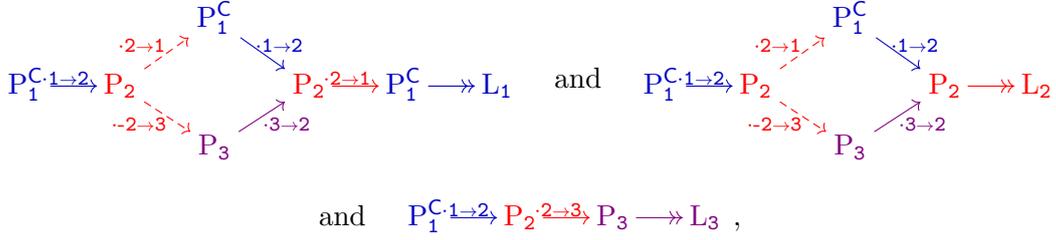
\[
\xy
(0,0)*{
\begin{tikzcd}[ampersand replacement=\&,row sep=tiny,column sep=small,arrows={shorten >=-.5ex,shorten <=-.5ex},labels={inner sep=.05ex}]
\phantom{.} \& \phantom{.} \& {\color{myblue}\lprojb[1]}\arrow[thin,myblue,->]{rd}[xshift=-.11cm]{\cdot\obstuff{1}\scalebox{.8}{$\to$}\obstuff{2}} \& \phantom{.} \& \phantom{.} \& \phantom{.}
\\
{\color{myblue}\lprojb[1]}\arrow[thin,myblue,->]{r}[xshift=-.11cm]{\cdot\obstuff{1}\scalebox{.8}{$\to$}\obstuff{2}} \& {\color{myred}\lproj[2]}\arrow[thin,densely dashed,myred,->]{ru}{\cdot\obstuff{2}\scalebox{.8}{$\to$}\obstuff{1}}\arrow[thin,densely dashed,myred,->]{rd}[swap]{\cdot\text{-}\obstuff{2}\scalebox{.8}{$\to$}\obstuff{3}} \& \phantom{.} \& {\color{myred}\lproj[2]}\arrow[thin,myred,->]{r}[xshift=-.11cm]{\cdot\obstuff{2}\scalebox{.8}{$\to$}\obstuff{1}} \& {\color{myblue}\lprojb[1]}\arrow[thin,myblue,->>]{r}{} \& {\color{myblue}\lsimple[1]}
\\
\phantom{.} \& \phantom{.} \& {\color{mypurple}\lproj[3]}\arrow[thin,mypurple,->]{ru}[swap]{\cdot\obstuff{3}\scalebox{.8}{$\to$}\obstuff{2}} \& \phantom{.} \& \phantom{.} \& \phantom{.}
\end{tikzcd}};
\endxy
\quad\text{and}\quad
\xy
(0,0)*{
\begin{tikzcd}[ampersand replacement=\&,row sep=tiny,column sep=small,arrows={shorten >=-.5ex,shorten <=-.5ex},labels={inner sep=.05ex}]
\phantom{.} \& \phantom{.} \& {\color{myblue}\lprojb[1]}\arrow[thin,myblue,->]{rd}[xshift=-.11cm]{\cdot\obstuff{1}\scalebox{.8}{$\to$}\obstuff{2}} \& \phantom{.} \& \phantom{.}
\\
{\color{myblue}\lprojb[1]}\arrow[thin,myblue,->]{r}[xshift=-.11cm]{\cdot\obstuff{1}\scalebox{.8}{$\to$}\obstuff{2}} \& {\color{myred}\lproj[2]}\arrow[thin,densely dashed,myred,->]{ru}{\cdot\obstuff{2}\scalebox{.8}{$\to$}\obstuff{1}}\arrow[thin,densely dashed,myred,->]{rd}[swap]{\cdot\text{-}\obstuff{2}\scalebox{.8}{$\to$}\obstuff{3}} \& \phantom{.} \& {\color{myred}\lproj[2]}\arrow[thin,myred,->>]{r}{} \& {\color{myred}\lsimple[2]}
\\
\phantom{.} \& \phantom{.} \& {\color{mypurple}\lproj[3]}\arrow[thin,mypurple,->]{ru}[swap]{\cdot\obstuff{3}\scalebox{.8}{$\to$}\obstuff{2}} \& \phantom{.} \& \phantom{.}
\end{tikzcd}};
\endxy
\]
\[
\text{and}\quad
\begin{tikzcd}[ampersand replacement=\&,row sep=tiny,column sep=small,arrows={shorten >=-.5ex,shorten <=-.5ex},labels={inner sep=.05ex}]
{\color{myblue}\lprojb[1]}\arrow[thin,myblue,->]{r}[xshift=-.11cm]{\cdot\obstuff{1}\scalebox{.8}{$\to$}\obstuff{2}} \& {\color{myred}\lproj[2]}\arrow[thin,myred,->]{r}[xshift=-.11cm]{\cdot\obstuff{2}\scalebox{.8}{$\to$}\obstuff{3}} 
\& {\color{mypurple}\lproj[3]}\arrow[thin,mypurple,->>]{r}{} \& {\color{mypurple}\lsimple[3]}
\end{tikzcd},
\]
where we calculate the resolutions 
for all simple $\bzigzag(\typeA[3])$-modules. 
(In this case all resolutions are finite, but of different length. 
We will see below that this is in fact always the case for the type $\typeA[]$ graphs 
with vertex condition imposed on one leaf.)
\end{example}

\subsection{The case of \texorpdfstring{$\zigzag$}{Z}}\label{subsec:koszul-zigzag}

\subsubsection{Construction}\label{subsubsec:koszul-construction}

We first show abstractly that $\zigzag$ is Koszul in case $\Gg$ 
is not a type $\typeADE$ graph. Then we construct the linear projective
resolution explicitly using a non-terminating algorithm motivated by Chebyshev polynomials. 

Before showing koszulity we make the following observation, 
which is inspired by \cite[Section 2]{b-cube-zero}. 
Let us write $\algstuff{Q}(\placeholder)$ for the projective cover of a module. 
If $\algstuff{K}_{0}$ is a $\zigzag$-module which has a radical filtration of length 
$2$ with multiplicity vectors $\underline{a}_{0}$ of $\algstuff{K}_{0}/\mathrm{rad}(\algstuff{K}_{0})$, 
respectively $\underline{b}_{0}$ of $\mathrm{rad}(\algstuff{K}_{0})/\mathrm{rad}^{2}(\algstuff{K}_{0})$, 
then we have a short exact sequence
$\algstuff{K}_{1}\hookrightarrow\algstuff{Q}(\algstuff{K}_{0})\twoheadrightarrow\algstuff{K}_{0}$,
where $\underline{a}_{1}=\amatrix\underline{a}_{0}-\underline{b}_{0}$, 
respectively $\underline{b}_{1}=\underline{a}_{0}$, 
are the corresponding multiplicity vectors of the radical filtration of $\algstuff{K}_{1}$.
Further, if $\algstuff{K}_{0}$ is graded, then so 
is $\algstuff{K}_{1}$. 

Assume that $\underline{a}_{1}\neq 0$.
If $\algstuff{K}_{0}$ 
is generated in degree $0$ and its radical filtration 
is equal to its grading filtration, then $\algstuff{K}_{1}$ is 
generated in degree $1$, and 
its radical filtration is also equal to its grading filtration. 
This holds since $\mathrm{rad}(\algstuff{K}_{1})$ lies 
inside $\mathrm{rad}^2(\algstuff{Q}(\algstuff{K}_{0}))$, 
hence is of degree $2$, and $\underline{a}_1$ agrees with the multiplicities 
of the kernel of $\algstuff{Q}(\algstuff{K}_{0})\twoheadrightarrow\algstuff{K}_{0}$ 
in degree $1$. Thus, the filtrations agree 
and $\algstuff{K}_{1}$ is generated in degree $1$.
Hence, as long as $\underline{a}_{t}\neq 0$, 
one can produce $\zigzag$-modules $\algstuff{K}_{t}$ for all $t\in\Z_{\geq 0}$ 
in the same manner having the same properties.

Now come some of our main players in this section, the \textit{Chebyshev polynomials} (of the second kind). 
They are defined via the recursion
\begin{gather}\label{eq:cp-recursion-formula}
U_{-1}(\xpar)=0,
\quad
U_{0}(\xpar)=1,
\quad
U_{t}(\xpar)=\xpar U_{t{-}1}(\xpar)-U_{t{-}2}(\xpar),
\text{ for }t\in\Z_{\geq 1}.
\end{gather}
Having the polynomials defined, observe that, 
for $t\in\Z_{\geq 0}$, the $t^{\mathrm{th}}$ multiplicity vectors of 
the radical filtration of $\algstuff{K}_{t}$ are given by
\begin{gather}\label{eq:cp-recursion}
\begin{pmatrix}
\underline{a}_{t}
\\
\underline{b}_{t}
\end{pmatrix}
=
\begin{pmatrix}
U_{t}(\amatrix) & -U_{t{-}1}(\amatrix)
\\
U_{t{-}1}(\amatrix) & -U_{t{-}2}(\amatrix)
\end{pmatrix}
\begin{pmatrix}
\underline{a}
\\
\underline{b}
\end{pmatrix},
\end{gather}
where we let $U_{-2}(\xpar)=0$ in case $t=0$.

\begin{proposition}\label{proposition:koszul-abstract}
If $\Gg$ is not a type $\typeADE$ graph, then
$\zigzag$ is Koszul.
\end{proposition}

Recall that a resolution as in \eqref{eq:koszul-reso} 
is called \textit{minimal} if no indecomposable summand of $\algstuff{Q}_{t}$ 
lies in the kernel for all $t\in\Z_{\geq 0}$.

\begin{proof}
We first note that the only graphs $\Gg$ such that $U_{t}(\amatrix)=0$ 
holds for some $t\in\Z_{\geq 0}$ are type $\typeADE$ graphs. 
This follows since the roots of 
the Chebyshev polynomial $U_{t}(X)$ are known to be all of the form  
$2\cos(k/(t+1)\pi)\in]-2,2[$ for $k\in\{1,\dots,t\}$, while, 
by \cite{sm-ADE} or \cite[Section 3.1.1]{bh-graphs}, 
$\Gg$ has a Perron--Frobenius eigenvalue $\lambda\geq 2$. 
This also implies that $U_{t}(\amatrix)$ is irreducible (meaning 
that, for all $\ii,\ii[j]$, there exists $N_{\ii\ii[j]}\in\Z_{\geq 0}$ 
such that the $\ii$-$\ii[j]$ position of $U_{t}(\amatrix)^{N_{\ii\ii[j]}}$ is in $\Z_{>0}$).
Further, that $U_{t}(\amatrix)$ has non-negative entries follows
from categorification, cf. \cite[End of Section 5.1]{mt-soergel} 
(which uses the polynomial $2\imatrix+\amatrix$). Hence, $U_{t}(\amatrix)$ is a 
symmetric, irreducible, non-negative matrix, 
and we can apply Perron--Frobenius theory to find 
an eigenvector of $U_{t}(\amatrix)$ with entries and associated 
eigenvalue from $\R_{>0}$. But having such 
an eigenvector implies that no columns or rows can be zero.
Thus, if we start with $\algstuff{K}_{0}=\lsimple[\ii]$, then $\underline{a}_{t}$ 
in \eqref{eq:cp-recursion} will never be the zero vector.
Summarized, to produce a minimal projective resolution of $\algstuff{K}_{0}=\lsimple[\ii]$ 
we successively resolve the module $\algstuff{K}_{t}$ with its projective cover 
$\algstuff{Q}(\algstuff{K}_{t})$ with kernel $\algstuff{K}_{t{+}1}$. By the arguments 
above, this minimal projective resolution will be linear, and the decomposition of 
$\algstuff{Q}(\algstuff{K}_{t})$ into indecomposables 
is given by the $\ii^{\mathrm{th}}$ column of $U_{t}(\amatrix)$.
\end{proof} 

We now construct the linear projective resolutions explicitly.

\begin{definition}\label{definition:koszul-resolution-graph}
Fix a graph $\Gg$.
A \textit{resolution graph} $\Rg=(V,E)$ associated to $\Gg$ is a directed graph, whose 
vertex set $V=\bigcup_{t\in\Z_{\geq 0}}V_{t}$ is a disjoint union of 
finite sets $V_{t}$ such that each vertex $v(\ii)$
is labeled by a vertex $\ii$ of $\Gg$. Moreover, the edge 
set $E=\bigcup_{t\in\Z_{\geq 1}}E_{t}$ 
is a disjoint union of finite sets $E_{t}$ such that 
$E_{t}$ contains only edges $e(z)\colon v(\ii)\to v(\ii[j])$ from $V_{t}$ to $V_{t{+}1}$ 
that are labeled by some $z\in\K$.
The two sets $V_{t},E_{t}$ are called the 
\textit{$t^{\mathrm{th}}$ level of $\Rg$}.

A level $\sshift$ resolution graph $\Rg_{\sshift}$ is the same data, but 
$V_{t}=\emptyset=E_{t}$ for all $t\in\Z_{\geq\sshift}$.

Further, if we fix $\ii$, then we also consider monochrome sets $V_{t{-}1}(\ii)=\{v(\ii)\in V_{t{-}1}\}$ 
(``colored'' only by $\ii$), 
$E_{t}(\ii)=\{e(z)\colon v(\ii)\to v(\ii[j])\mid v(\ii)\in V_{t{-}1}\}$ and 
$V_{t}(\ii)=\{v(\ii[j])\mid v(\ii[j])\text{ is a target of }e(z)\in E_{t}(\ii)\}$, and the graph 
$\Rg_{t}(\ii)=(V_{t{-}1}(\ii)\cup V_{t}^{\ii},E_{t}(\ii))$.
Denote by $\Rg_{t}(\ii)=\bigcup_{r}\Rg_{t}(\ii,r)$ its decomposition 
into connected components $\Rg_{t}(\ii,r)=(V_{t{-}1}(\ii,r)\cup V_{t}(\ii,r),E_{t}(\ii,r))$ 
(seen as an unoriented graph).
\end{definition}

A \textit{successor of a vertex $v(\ii)$} in a directed graph is a vertex $v(\ii[j])$ 
such that there is a directed path from $v(\ii)$ to $v(\ii[j])$. Similarly 
in case of successors of a set of vertices.

\begin{definition}\label{definition:koszul-weighting}
A \textit{$\K$-weighting} of a resolution graph and a fixed set of vertices $V^{\prime}$ of it
are elements $b_{v(\ii[j])}\in\K$ for all successors $v(\ii[j])$ of $V^{\prime}$ such that 
for each vertex $v(\ii)\in X$ we have
\[
{\textstyle\sum_{v(\ii[j])}}\;
z_{v(\ii[j])}b_{v(\ii[j])}=0,
\]
where $z_{v(\ii[j])}\in\K$ is the label of $e(z)\colon v(\ii)\to v(\ii[j])$, 
and the sum is over all successors $v(\ii[j])$ of $v(\ii)$. We denote the 
vector space of all $\K$-weightings of $V^{\prime}$ by $\mathrm{We}(V^{\prime})$.
\end{definition}

\IncMargin{1em}
\begin{algorithm}
\SetKwData{Left}{left}\SetKwData{This}{this}\SetKwData{Up}{up}
\SetKwFunction{Union}{Union}\SetKwFunction{FindCompress}{FindCompress}
\SetKwInOut{Input}{input}\SetKwInOut{Output}{output}

 \Input{a graph $\Gg$ and a fixed vertex $\ii$ of it\;}
 \Output{a resolution graph $\Rg=\Rg(\ii)=(\bigcup_{s}V_{\sshift},\bigcup_{s}E_{\sshift})$ for $\Gg$\;}
 \BlankLine
 \textit{initialization ($\sshift=0$)}, let $\Rg_{0}=(V_{0},E_{0})$ be the 
 level $0$ resolution graph with $V_{0}=\{\ii\}$, $E_{0}=\emptyset$, 
 and set $V_{t}=\emptyset=E_{t}$ for all $t\in\Z_{\geq -1}$, $t\neq 0$\;
 \For{$\sshift\in\Z_{\geq 1}$}{
 \tcc*[h]{recall that the vertices of $\Gg$ are numbered, say $\{\ii[1],\dots,\ii[n]\}$}\;
 \For{$\ii[j]=\ii[1]$ \KwTo $\ii[n]$}{
 \For{$\Rg_{\sshift}^{\ii}(r)$ connected component}{
 \tcc*[h]{the potential solutions}\;
 fix a basis $\mathbb{B}(\sshift,\ii,r)$ of $\mathrm{We}(V_{\sshift{-}1}(\ii,r))$\;
 {\tcc*[h]{add them to $\Rg_{\sshift}$; called fork moves}\;
 \For{$b=(b_{1},\dots,b_{l})\in\mathbb{B}(\sshift,\ii,r)$}
 {add a vertex $v(\ii)$ to $V_{\sshift{+}1}$ 
 and edges $e(b_{k})\colon v(\ii[j]_{k}) \to v(\ii)$ to $E_{\sshift}$\;}
 }}}
 \tcc*[h]{singleton moves}\;
 \For{$v(\ii[j])\in V_{\sshift{-}1}$}
 {add a vertex $v(\ii[k])$ to $V_{\sshift}$ for all neighbors of $\ii[j]$ in $\Gg$ which are not neighbors of $v(\ii[k])$ in $\Rg$, 
 and an edge $e(1)\colon v(\ii[k])\to v(\ii[j])$ to $E_{\sshift}$\;}
 }
 \caption{The Chebyshev algorithm a.k.a. producing 
 linear projective resolutions.}\label{algorithm:koszul}
\end{algorithm}
\DecMargin{1em}

\fullref{algorithm:koszul} is, by birth, well-defined and depends on some choices, 
since in each step there are a few systems of linear equations 
(with unknowns $b_{v(\ii[j])}$) one needs to solve 
and choose a basis for its solution space. However, we will see in 
\fullref{lemma:koszul-multiplicities} that the vertex sets are independent of the involved choices. 
Note further that \fullref{algorithm:koszul} usually does not terminate, 
and output is to be understood that 
we can stop the algorithm at any level and get arbitrary long 
parts of the resolutions.

\begin{example}\label{example:algorithm}
In this example, for readability, we omit the orientation of the edges 
(we always read left to right and underline the starting vertex). 
We also write $\ii$ short for $v(\ii)$ in illustrations, and if we do not specify 
the edge labels, then they are $1$, by convention, while we write $-$ instead of $-1$ for short.

The first example is a type $\atypeA[2]$ graph where we choose 
the vertex $\ii[0]$. Then one gets
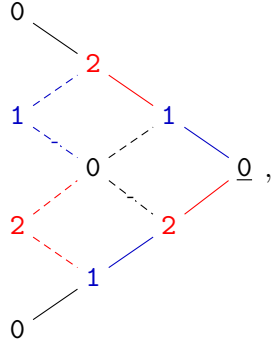
\[
\xy
(0,0)*{
\begin{tikzcd}[ampersand replacement=\&,row sep=tiny,column sep=small,arrows={shorten >=-.5ex,shorten <=-.5ex},labels={inner sep=.05ex}]
\ii[0]\arrow[thin,-]{rd}{} \& \phantom{.} \& \phantom{.} \& \phantom{.}
\\
\phantom{.} \& {\color{myred}\ii[2]}\arrow[thin,myred,-]{rd}{} \& \phantom{.} \& \phantom{.}
\\
{\color{myblue}\ii[1]}\arrow[thin,densely dashed,myblue,-]{rd}[description]{\text{-}}\arrow[thin,densely dashed,myblue,-]{ru}{} \& \phantom{.} \& {\color{myblue}\ii[1]}\arrow[thin,myblue,-]{rd}{} \& \phantom{.}
\\
\phantom{.} \& \ii[0]\arrow[thin,densely dashed,-]{rd}[description]{\text{-}}\arrow[thin,densely dashed,-]{ru}{} \& \phantom{.} \& \underline{\ii[0]}
\\
{\color{myred}\ii[2]}\arrow[thin,densely dashed,myred,-]{rd}{}\arrow[thin,densely dashed,myred,-]{ru}{} \& \phantom{.} \& {\color{myred}\ii[2]}\arrow[thin,myred,-]{ru}{} \& \phantom{.}
\\
\phantom{.} \& {\color{myblue}\ii[1]}\arrow[thin,myblue,-]{ru}{} \& \phantom{.} \& \phantom{.}
\\
\ii[0]\arrow[thin,-]{ru}{} \& \phantom{.} \& \phantom{.} \& \phantom{.}
\end{tikzcd}};
\endxy,
\]
where the dashed lines come from 
fork moves, while the straight lines come from singleton moves, cf. \fullref{example-koszul-affineA}. 
Note that all $\K$-weighting spaces are one-dimensional in this case, 
and the above is up to scaling unique.
 
Let us now do an example with two very different choices along the way. 
So let $\Gg$ be a type $\typeD[5]$ graph with starting vertex ${\color{myred}2}$, and we get
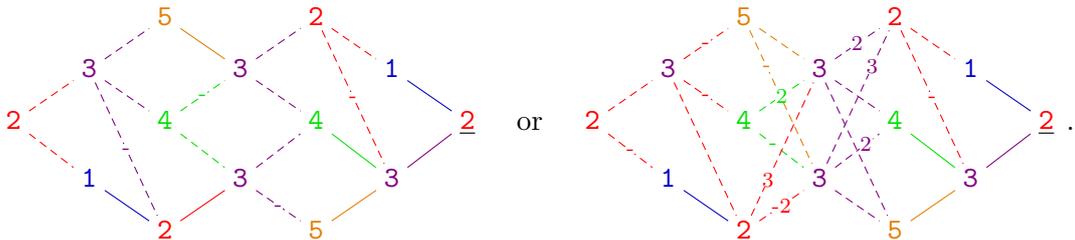
\[
\xy
(0,0)*{
\begin{tikzcd}[ampersand replacement=\&,row sep=tiny,column sep=small,arrows={shorten >=-.5ex,shorten <=-.5ex},labels={inner sep=.05ex}]
\phantom{.} \& \phantom{.} \& {\color{myorange}\ii[5]}\arrow[thin,myorange,-]{rd}{} \& \phantom{.} \& {\color{myred}\ii[2]}\arrow[thin,densely dashed,myred,-]{rd}{}\arrow[thin,densely dashed,myred,-]{rddd}[description]{\text{-}} \& \phantom{.} \& \phantom{.}
\\
\phantom{.} \& {\color{mypurple}\ii[3]}\arrow[thin,densely dashed,mypurple,-]{rd}{}\arrow[thin,densely dashed,mypurple,-]{ru}{}\arrow[thin,densely dashed,mypurple,-]{rddd}[description]{\text{-}}  \& \phantom{.} \& {\color{mypurple}\ii[3]}\arrow[thin,densely dashed,mypurple,-]{rd}{}\arrow[thin,densely dashed,mypurple,-]{ru}{} \& \phantom{.} \& {\color{myblue}\ii[1]}\arrow[thin,myblue,-]{rd}{} \& \phantom{.}
\\
{\color{myred}\ii[2]}\arrow[thin,densely dashed,myred,-]{rd}{}\arrow[thin,densely dashed,myred,-]{ru}{} \& \phantom{.} \& {\color{mygreen}\ii[4]}\arrow[thin,densely dashed,mygreen,-]{ru}[description]{\text{-}}\arrow[thin,densely dashed,mygreen,-]{rd}{} \& \phantom{.} \& {\color{mygreen}\ii[4]}\arrow[thin,mygreen,-]{rd} \& \phantom{.} \& \underline{{\color{myred}\ii[2]}}
\\
\phantom{.} \& {\color{myblue}\ii[1]}\arrow[thin,myblue,-]{rd}{} \& \phantom{.} \& {\color{mypurple}\ii[3]}\arrow[thin,densely dashed,mypurple,-]{ru}{}\arrow[thin,densely dashed,mypurple,-]{rd}[description]{\text{-}} \& \phantom{.} \& {\color{mypurple}\ii[3]}\arrow[thin,mypurple,-]{ru}{} \& \phantom{.}
\\
\phantom{.} \& \phantom{.} \& {\color{myred}\ii[2]}\arrow[thin,myred,-]{ru}{} \& \phantom{.} \& {\color{myorange}\ii[5]}\arrow[thin,myorange,-]{ru}{} \& \phantom{.} \& \phantom{.}
\end{tikzcd}};
\endxy
\quad\text{or}\quad
\xy
(0,0)*{
\begin{tikzcd}[ampersand replacement=\&,row sep=tiny,column sep=small,arrows={shorten >=-.5ex,shorten <=-.5ex},labels={inner sep=.05ex}]
\phantom{.} \& \phantom{.} \& {\color{myorange}\ii[5]}\arrow[thin,densely dashed,myorange,-]{rd}{}\arrow[thin,densely dashed,myorange,-]{rddd}[description,near start]{\text{-}} \& \phantom{.} \& {\color{myred}\ii[2]}\arrow[thin,densely dashed,myred,-]{rd}{}\arrow[thin,densely dashed,myred,-]{rddd}[description]{\text{-}} \& \phantom{.} \& \phantom{.}
\\
\phantom{.} \& {\color{mypurple}\ii[3]}\arrow[thin,densely dashed,myred,-]{ru}[description]{\text{-}}\arrow[thin,densely dashed,myred,-]{rd}[description]{\text{-}}\arrow[thin,densely dashed,myred,-]{rddd}{} \& \phantom{.} \& {\color{mypurple}\ii[3]}\arrow[thin,densely dashed,mypurple,-]{ru}[description]{2}\arrow[thin,densely dashed,mypurple,-]{rd}{}\arrow[thin,densely dashed,mypurple,-]{rddd}{} \& \phantom{.} \& {\color{myblue}\ii[1]}\arrow[thin,myblue,-]{rd}{} \& \phantom{.}
\\
{\color{myred}\ii[2]}\arrow[thin,densely dashed,myred,-]{rd}[description]{\text{-}}\arrow[thin,densely dashed,myred,-]{ru}{} \& \phantom{.} \& {\color{mygreen}\ii[4]}\arrow[thin,densely dashed,mygreen,-]{ru}[description]{2}\arrow[thin,densely dashed,mygreen,-]{rd}[description,near start]{\text{-}} \& \phantom{.} \& {\color{mygreen}\ii[4]}\arrow[thin,mygreen,-]{rd} \& \phantom{.} \& \underline{{\color{myred}\ii[2]}}
\\
\phantom{.} \& {\color{myblue}\ii[1]}\arrow[thin,myblue,-]{rd}{} \& \phantom{.} \& {\color{mypurple}\ii[3]}\arrow[thin,densely dashed,mypurple,-]{rd}{}\arrow[thin,densely dashed,mypurple,-]{ru}[description,near end]{2}\arrow[thin,densely dashed,mypurple,-]{ruuu}[description,near end]{3} \& \phantom{.} \& {\color{mypurple}\ii[3]}\arrow[thin,mypurple,-]{ru}{} \& \phantom{.}
\\
\phantom{.} \& \phantom{.} \& {\color{myred}\ii[2]}\arrow[thin,densely dashed,myred,-]{ru}[description]{\text{-}2}\arrow[thin,densely dashed,myred,-]{ruuu}[description,near start]{3} \& \phantom{.} \& {\color{myorange}\ii[5]}\arrow[thin,myorange,-]{ru}{} \& \phantom{.} \& \phantom{.}
\end{tikzcd}};
\endxy.
\]
Note the following two crucial observation: First, although we made quite different choices 
above, the multiplicities in each column are equal, and given by the Chebyshev recursion 
\eqref{eq:cp-recursion}. Second, \fullref{algorithm:koszul} 
terminates for the finite type $\typeADE$ graphs, but not for the affine 
ones. We will see that both is always the case.
\end{example}

If we have a minimal linear projective resolution as in \eqref{eq:koszul-reso} 
which stops at $\algstuff{Q}_{\sshift}$ for some $\sshift\in\Z_{\geq 0}$, 
then we say its of length $\sshift$.

\begin{definition}\label{definition:koszul-algo-output}
Let $\sshift\in\Z_{\geq 0}$, and let $\Rg_{\sshift}=(\bigcup_{t}V_{t},\bigcup_{t}E_{t})$ be 
the output of \fullref{algorithm:koszul} of which we assume that 
it has not terminated. Then we define 
a complex of $\zigzag$-modules:
\smallskip
\begin{enumerate}

\setlength\itemsep{.15cm}

\renewcommand{\theenumi}{(\ref{definition:koszul-algo-output}.a)}
\renewcommand{\labelenumi}{\theenumi}

\item \label{enum:koszul1} For $0\leq t\leq\sshift$ we let $\algstuff{Q}_{t}=\bigoplus_{v(\ii)\in V_{t}}\lproj[i]$, i.e.
we identify the vertices $v(\ii)$ in the $t^{\mathrm{th}}$ level with the projectives
$\zigzag$-modules $\lproj[i]$, and take direct sums.

\renewcommand{\theenumi}{(\ref{definition:koszul-algo-output}.b)}
\renewcommand{\labelenumi}{\theenumi}

\item \label{enum:koszul2} For $0\leq t\leq\sshift{-}1$ we let 
$\phi\colon\algstuff{Q}_{t{-}1}\to\algstuff{Q}_{t}=\bigoplus_{e(z)\colon v(\ii)\to v(\ii[j])\in E_{t}}(\cdot z\pathx{i}{j})$, i.e.
we identify the edges $e(z)\colon v(\ii)\to v(\ii[j])$ in the $t^{\mathrm{th}}$ level with the
$\zigzag$-equivariant maps given by post-composition with $z\pathx{i}{j}$, and take matrices.

\end{enumerate}
\smallskip

We write $\mathrm{Reso}(\Rg_{\sshift})$ for this complex.
\end{definition}

\begin{lemma}\label{lemma:koszul-multiplicities}
Assume \fullref{algorithm:koszul} does not terminate before the $\sshift^{\mathrm{th}}$ level, 
and $\mathrm{Reso}(\Rg_{\sshift})$ is a minimal linear projective resolution of length $\sshift$ 
with multiplicities given by \eqref{eq:cp-recursion}. If $U_{\sshift}(\amatrix)\neq 0$, 
then $\mathrm{Reso}(\Rg_{\sshift{+}1})$ is minimal linear projective 
resolution of length $\sshift+1$.
\end{lemma}

\begin{proof}
Because $U_{\sshift}(\amatrix)\neq 0$ we know by the same arguments as above 
that the kernel $\algstuff{K}_{\sshift}$ in generated in degree $1$. 
Thus, its degree $1$ part $\algstuff{K}_{\sshift}^{1}$ is a semisimple
$\zigzag$-module. Because $\algstuff{K}_{\sshift}^{1}$ is 
semisimple, the next step of the resolution is determined by choosing a basis 
for this $\zigzag$-module, which is precisely what \fullref{algorithm:koszul} does.
\end{proof}

\begin{lemma}\label{lemma:koszul-terminate}
\fullref{algorithm:koszul} terminates if and only if $\Gg$ is a type $\typeADE$ graph.
\end{lemma}

\begin{proof}
This is a direct consequence of \fullref{lemma:koszul-multiplicities} 
since (as we have already seen above) $U_{t}(\amatrix)$ contains a
zero column for some $t\in\Z_{\geq 0}$ if and only if $U_{t}(\amatrix)=0$ 
for some $t\in\Z_{\geq 0}$ if and only if $\Gg$ is a type $\typeADE$ graph 
(with $t+1$ being the Coxeter number of $\Gg$). 
\end{proof}

\begin{proposition}\label{proposition:koszul-the-reso}
If \fullref{algorithm:koszul} does not terminate, then it produces 
a minimal linear projective resolution of the $\zigzag$-module $\lsimple$ 
with multiplicities given by \eqref{eq:cp-recursion}.
\end{proposition}

\makeautorefname{lemma}{Lemmas}

\begin{proof}
This is a direct consequence of 
\fullref{lemma:koszul-multiplicities} and \ref{lemma:koszul-terminate}.
\end{proof}

\makeautorefname{lemma}{Lemma}

\subsubsection{Elimination}\label{subsubsec:koszul-elim}

Here is the numerical condition which we are going to use.

\begin{lemma}\label{lemma:the-cartan-koszul}
If $\zigzag$ is Koszul, then 
its graded Cartan matrix $\cmatrix_{\qpar}$ is invertible 
in the ring of matrices with entries from $\Z[[\qpar]]$.
Moreover, the column sums of $\cmatrix_{\qpar}^{-1}$ are power series 
in $\Z[[\qpar]]$ of the from
\[
{\textstyle\sum_{\sshift=0}^{\infty}}\,(-1)^{\sshift}a_{\sshift}\qpar^{\sshift}
\]
with coefficients $a_{\sshift}\in\Z_{\geq 1}$.
\end{lemma}

\begin{proof}
The statement about the invertibility 
of the graded Cartan matrix is the usual consequence 
of koszulity, cf. \cite[Theorem 2.11.1]{bgs-koszul}. For the 
second statement note that the inverse 
of the graded Cartan matrix encodes the 
graded multiplicities of the resolution of the simples 
by projectives.
\end{proof}

For a formal power series $f=\sum_{i=0}^{\infty}b_{i}\qpar^{i}\in\Z[[\qpar]]$ 
we say that $f$ has \textit{gaps of 
size $k\in\Z_{\geq 0}$} if there exist $i\leq j\in\Z_{\geq 0}$ such that $j-i+1=k$ and 
$b_{i}=b_{i+1}=\cdots=b_{j-1}=b_{j}=0$.

Recall that $n$ denotes the number of vertices of $\Gg$.

\begin{lemma}\label{lemma:the-cartan-koszul-b}
If $\zigzag$ is Koszul, 
then $(\mathrm{det}(\cmatrix_{\qpar}))^{-1}$ does not have 
gaps of size $>2n-2$.
\end{lemma}

\begin{proof}
Recall that the matrix of cofactors $A^{\ast}$ of a fixed 
matrix $A$ is determined by $AA^{\ast}=\mathrm{det}(A)\imatrix$. By construction, 
the matrix $\cmatrix_{\qpar}^{\ast}$ has its entries in $\Z[\qpar]$. 
In fact, the entries of $\cmatrix_{\qpar}^{\ast}$ are 
polynomials in $\Z[\qpar]$ of degree 
at most $2n-2$ (this follows from \fullref{proposition:cartan-zigzag}), and the 
statement from \fullref{lemma:the-cartan-koszul} boils down to 
the graded Cartan determinant being invertible with the claimed property. To be precise,
if the $j^{\mathrm{th}}$ column sum of $\cmatrix_{\qpar}^{\ast}$ 
is of the form $\sum_{\sshift=0}^{2n-2}b_{\sshift}\qpar^{\sshift}$, 
and the inverse of the graded 
Cartan determinant is $\sum_{\sshift=0}^{\infty}b_{\sshift}^{\prime}\qpar^{\sshift}$, then
$({\textstyle\sum_{\sshift=0}^{2n-2}}\,b_{\sshift}\qpar^{\sshift})
({\textstyle\sum_{\sshift=0}^{\infty}}\,b_{\sshift}^{\prime}\qpar^{\sshift})
=
{\textstyle\sum_{\sshift=0}^{\infty}}\,(-1)^{\sshift}a_{\sshift}\qpar^{\sshift}$
implies $(-1)^{k}a_{k}=b_{0}b_{k}^{\prime}+\cdots+b_{2n-2}b_{k-2n+2}^{\prime}$. Thus, 
$a_{j}=0$ for some $j\in\Z_{\geq 0}$ if there 
would be a gap of size $>2n-2$.
\end{proof}

\begin{proposition}\label{proposition:finite-ADE-koszul}
If $\Gg$ is a type $\typeADE$ graph, then $\zigzag$ is not Koszul.
\end{proposition}

\begin{proof}
We want to use \fullref{lemma:the-cartan-koszul-b} to show that 
$\zigzag$ is not Koszul. To this end, observe that 
we already know from \eqref{eq:the-data}
the corresponding graded Cartan determinants. Their formal inverses are not hard to 
compute:
\begin{gather*}
\begin{aligned}
\typeA\colon& (1-\qpar^{2}){\textstyle\sum_{\sshift=0}^{\infty}}\,\qpar^{(2n+2)\sshift},
\quad&&\text{gap}=2n-1,
\\
\typeD,n\text{ even}\colon& (1-\qpar^{2}\pm\cdots+\qpar^{2n-4})
{\textstyle\sum_{\sshift=0}^{\infty}}\,(-1)^{\sshift}(\sshift+1)\qpar^{(2n-2)\sshift},\quad&&\text{gap}=1,
\\
\typeD,n\text{ odd}\colon& (1-\qpar^{2}\pm\cdots-\qpar^{2n-4})
{\textstyle\sum_{\sshift=0}^{\infty}}\qpar^{(4n-4)\sshift},\quad&&\text{gap}=2n-1,
\\
\typeE[6]\colon& (1-\qpar^{2}+\qpar^{4}-\qpar^{8}+\qpar^{10}-\qpar^{12})
{\textstyle\sum_{\sshift=0}^{\infty}}\,\qpar^{24\sshift},\quad&&\text{gap}=11,
\\
\typeE[7]\colon& (1-\qpar^{2}+\qpar^{4}){\textstyle\sum_{\sshift=0}^{\infty}}\,(-1)^{\sshift}\qpar^{18\sshift},
\quad&&\text{gap}=13,
\\
\typeE[8]\colon& (1-\qpar^{2}+\qpar^{4}+\qpar^{10}-\qpar^{12}+\qpar^{14})
{\textstyle\sum_{\sshift=0}^{\infty}}\,(-1)^{\sshift}\qpar^{30\sshift},\quad&&\text{gap}=15.
\end{aligned}
\end{gather*}
Hence, except for the type $\typeD[n]$ graph with $n=2m$, the gap equals $2n-1$.  
To rule out the remaining cases, 
first recall the various graded Cartan determinants, cf. in the proof of \fullref{lemma:bipartite}, and
we make the following claim.
\medskip

\textit{Claim.} The graded determinant of 
every $n-1$ minor of $\cmatrix_{\qpar}(\zigzag(\typeD[2m]))$ is divisible by $\qnumber{2}$.
\medskip
 
Before we prove this claim, let us state the consequences. This means that we can simplify
\[
\cmatrix_{\qpar}\cmatrix_{\qpar}^{\ast}=(1+\qpar^{2n{-}2})\qnumber{2}\imatrix
\quad\rightsquigarrow\quad
\cmatrix_{\qpar}\ccmatrix_{\qpar}^{\ast}=(1+\qpar^{2n{-}2})\imatrix,
\]
where $\ccmatrix_{\qpar}^{\ast}=(\qnumber{2})^{-1}\cmatrix_{\qpar}^{\ast}$ 
is a matrix with entries in $\Z[\qpar]$ 
of degree at most $2n-4$. But now the argument using gaps applies 
again since $(1+\qpar^{2n{-}2})^{-1}=\sum_{\sshift=0}^{\infty}(-1)^{\sshift}\qpar^{(2n-2)\sshift}$ 
(the gap is of size $2n-3$, which is strictly greater than $2n-4$). Hence,
this case is ruled out as well, and it remains to prove the claim.
\medskip

\textit{Proof of the claim.} We prove the statement 
by induction on the number of vertices. For the type $\typeD[4]$ graph the 
matrix of cofactors is easy to compute, i.e. we have
\[
\cmatrix_{\qpar}(\zigzag(\typeD[4]))^{\ast}
=
\qnumber{2}
\scalebox{.8}{$\displaystyle
\left(
\begin{array}{cccc}
1+\qpar^{4} & -\qpar\qnumber{2} & \qpar^{2} & \qpar^{2}
\\
-\qpar\qnumber{2} & (\qnumber{2})^{2} & -\qpar\qnumber{2} & -\qpar\qnumber{2}
\\
\qpar^{2} & -\qpar\qnumber{2} & 1+\qpar^{4} & \qpar^{2}
\\
\qpar^{2} & -\qpar\qnumber{2} & \qpar^{2} & 1+\qpar^{4}
\end{array}
\right)
$}.
\]
So assume that $n=2m>4$, and take a $i$-$j$ minor of $\cmatrix_{\qpar}(\zigzag(\typeD[n]))$, obtained by
erasing the $i^{\mathrm{th}}$ row and the $j^{\mathrm{th}}$ column. 
By symmetry, it suffices to consider the case $i\leq j$. As long as $i\leq n-2$ 
it turns out the determinant of this minor is equal to the 
determinant of $\cmatrix_{\qpar}(\zigzag(\typeA[i{-}1]))$ times the determinant of a $n-i$ minor 
of $\cmatrix_{\qpar}(\zigzag(\typeD[n{-}i{+}1]))$. But either $\mathrm{det}(\cmatrix_{\qpar}(\zigzag(\typeA[i{-}1])))=\qnumber{i}$ is divisible 
by $\qnumber{2}$, in case $i$ is even, or we know by induction that $\mathrm{det}(\cmatrix_{\qpar}(\zigzag(\typeD[n{-}i{+}1])))$, 
is in case $i$ is odd. Thus, three cases remain, i.e. $i=j=n-1$, $i=j=n$ and $i=n-1$ and $j=n$. 
The first two cases are easy since the determinant of the minor is equal to 
$\mathrm{det}(\cmatrix_{\qpar}(\zigzag(\typeA[n{-}1])))=\qnumber{n}$. For the remaining case 
one first expands the minor along the last column, followed by an expansion in the last row 
to determine that this minor is equal to 
$\pm\mathrm{det}(\cmatrix_{\qpar}(\zigzag(\typeA[n{-}3])))=\qnumber{(n{-}2)}$.
\end{proof}

\subsubsection{The proof}\label{subsubsec:koszul-proof}

We collect the all statements from above.

\begin{proof}[Proof of \fullref{theorem:koszul}]
\fullref{proposition:finite-ADE-koszul} shows that 
$\zigzag$ is not Koszul for $\Gg$ being a type $\typeADE$ graph, 
while \fullref{proposition:koszul-the-reso} constructs the linear 
projective resolution in all other cases, where we use the fact that 
having a linear projective resolution of length $\sshift$ for 
all $\sshift\in\Z_{\geq 0}$ implies koszulity. 
\end{proof}

\subsection{The case of \texorpdfstring{$\bzigzag$}{bZ}}\label{subsec:koszul-bzigzag}

\subsubsection{Construction -- Part I}\label{subsubsec:koszul-construction-b}

We first observe again how the successive kernels 
in a resolution of a $\bzigzag$-module with a two step 
radical filtration changes. 

For this purpose, we use the same notation and argumentation as above.
In particular, we get a short exact sequence 
$\algstuff{K}_{1}^{\bvec}\hookrightarrow\algstuff{Q}
(\algstuff{K}_{0}^{\bvec})\twoheadrightarrow\algstuff{K}_{0}^{\bvec}$. 
The crucial difference is that the 
multiplicity vectors of the 
radical filtration of $\algstuff{K}_{1}^{\bvec}$ are now 
$\underline{a}_{1}=\amatrix\underline{a}_{0}-\underline{b}_{0}$ 
and $\underline{b}_{1}=(\imatrix-\ematrix_{\bvec})\underline{a}_{0}$, 
where $\ematrix_{\bvec}$ as before denotes the diagonal matrix 
with only non-zero entries equal to $1$ in the $\ii[c]$-$\ii[c]$ position. 
Hence, as before, we can produce $\bzigzag$-modules 
$\algstuff{K}_{t}^{\bvec}$ for all $t\in\Z_{\geq 0}$ as long as $\underline{a}_{t}\neq 0$ which are generated in degree $t$ and have radical and grading filtrations that agree.

Note that, we do not have a pure polynomial recursion 
due to the occurrence of the matrix $\ematrix_{\bvec}$, 
and we have to modify our arguments.
To this end, we introduce a \textit{recursion of Chebyshev polynomials 
with matrix coefficients}, namely
\begin{gather}\label{eq:cpb-recursion-formula}
U^{\bvec}_{-1}(\xmpar)=0,
\;\;
U^{\bvec}_{0}(\xmpar)=\imatrix,
\;\;
U^{\bvec}_{t}(\xmpar)=\xmpar U^{\bvec}_{t{-}1}(\xmpar)
-(\imatrix-\ematrix_{\bvec})U^{\bvec}_{t{-}2}(\xmpar),
\text{ for }t\in\Z_{\geq 1}.
\end{gather}

\begin{remark}\label{remark:two-variable-CP}
One could state $U^{\bvec}_{t}(\xmpar)$ in terms of a polynomial
in two non-commuting variables $\xmpar,\ympar$. But we only need 
the version with $\ympar=(\imatrix-\ematrix_{\bvec})$, so we stick 
with it.
\end{remark}

Having these, observe that for $t\in\Z_{\geq 0}$, the $t^{\mathrm{th}}$ multiplicity vectors of 
the radical filtration of $\algstuff{K}_{t}^{\bvec}$ are then given by
\begin{gather}\label{eq:matrixcp-recursion}
\begin{pmatrix}
\underline{a}_{t}
\\
\underline{b}_{t}
\end{pmatrix}
=
\begin{pmatrix}
U^{\bvec}_{t}(\amatrix) 
& -U^{\bvec}_{t{-}1}(\amatrix)
\\
(\imatrix-\ematrix_{\bvec})U^{\bvec}_{t{-}1}(\amatrix) 
& -(\imatrix-\ematrix_{\bvec})U^{\bvec}_{t{-}2}(\amatrix)
\end{pmatrix}
\begin{pmatrix}
\underline{a}_0
\\
\underline{b}_0
\end{pmatrix}.
\end{gather}

We now express the $U^{\bvec}_{t}(\xmpar)$ in term of 
the usual 
Chebyshev polynomials 
$U_{t}(\xmpar)$ viewed as polynomials with matrix coefficients. 
To state it, we need the set of length $k$ compositions  
$C(t,k)=\{(i_{1},\dots,i_{k})\mid i_{j} 
\in\Z_{\geq 0},i_{1}+\dots+i_{k}=t+2-2k\}$ 
for $k,t\in\Z_{\geq 0}$.

\begin{lemma}\label{lemma:relation-polynomials}
For $t\in\Z_{\geq 0}$ it holds
\begin{gather}\label{eq:equality-polynomials}
U^{\bvec}_{t}(\xmpar)={\textstyle\sum_{k=1}^{r(t)}}\, 
{\textstyle\sum_{\underline{i}\in C(t,k)}}\,U_{i_{1}}(\xmpar) 
\ematrix_{\bvec}U_{i_{2}}(\xmpar)\ematrix_{\bvec} 
\dots \ematrix_{\bvec} U_{i_{k}}(\xmpar),
\end{gather}
where $r(t)=\lfloor t/2\rfloor +1$.
\end{lemma}

Note that, in case $\ematrix_{\bvec}$ would be the zero 
matrix, \eqref{eq:equality-polynomials} gives 
$U^{\bvec}_{t}(\xmpar)=U_{t}(\xmpar)$.

\begin{proof}
We prove the statement by induction. For $t=0$ and $t=1$ the 
equality holds. (In this case 
$U^{\bvec}_{t}(\xmpar) = U_{t}(\xmpar)$.) To show 
equality for $t\geq 2$, denote by $W_{t}(\xmpar)$ the 
right-hand side of \eqref{eq:equality-polynomials} 
for fixed $t$ and we verify it satisfies the same 
recursion as $U^{\bvec}_{t}(\xmpar)$, i.e. 
\[
W_{t}(\xmpar)=\xmpar W_{t{-}1}(\xmpar)-(\imatrix-\ematrix_{\bvec}) W_{t{-}2}(\xmpar).
\]
Next, let $U_{i_{1}}(\xmpar)\ematrix_{\bvec} 
U_{i_{2}}(\xmpar)\ematrix_{\bvec}\dots\ematrix_{\bvec}U_{i_{k}}(\xmpar)$ 
be a summand of $W_t(\xmpar)$ with ${\underline{i} \in C(t,k)}$ 
for some $1\leq k\leq r(t)$. We distinguish three cases.
\medskip

\noindent\textit{Case $i_{1}\geq 2$.} For this 
we have that $U_{i_{1}-1}(\xmpar)\ematrix_{\bvec} 
U_{i_{2}}(\xmpar)\ematrix_{\bvec}\dots\ematrix_{\bvec}U_{i_{k}}(\xmpar)$ is 
a summand of $W_{t-1}(\xmpar)$ for the sequence of indices 
in $C(t{-}1,k)$ and $U_{i_{1}-2}(\xmpar)\ematrix_{\bvec} 
U_{i_{2}}(\xmpar)\ematrix_{\bvec}\dots\ematrix_{\bvec} 
U_{i_{k}}(\xmpar)$ is a summand of $W_{t{-}2}(\xmpar)$ 
for the sequence of indices in $C(t{-}2,k)$. Hence, we have 
\[
U_{i_{1}}(\xmpar)\ematrix_{\bvec}\dots 
\ematrix_{\bvec} U_{i_{k}}(\xmpar)=\xmpar  
\cdot\left(U_{i_{1}{-}1}(\xmpar)\ematrix_{\bvec} 
\dots\ematrix_{\bvec}U_{i_{k}}(\xmpar)\right) 
-U_{i_{1}{-}2}(\xmpar)\ematrix_{\bvec}\dots\ematrix_{\bvec}U_{i_{k}}(\xmpar)
\]
by the Chebyshev recursion for the first factor.
\medskip

\noindent\textit{Case $i_{1}=1$.} In this 
case we only have that $U_{0}(\xmpar) 
\ematrix_{\bvec}U_{i_{2}}(\xmpar)\ematrix_{\bvec} 
\dots\ematrix_{\bvec}U_{i_{k}}(\xmpar)$ is a summand 
of $W_{t{-}1}(\xmpar)$ for the sequence of indices in 
$C(t{-}1,k)$. We obtain
\[
U_{i_{1}}(\xmpar)\ematrix_{\bvec} \dots \ematrix_{\bvec} 
U_{i_{k}}(\xmpar)=\xmpar\cdot\left(U_{i_{1}{-}1}(\xmpar) 
\ematrix_{\bvec}\dots\ematrix_{\bvec} U_{i_{k}}(\xmpar)\right),
\]
by using that $U_{1}(\xmpar)=\xmpar U_{0}(\xmpar)$.
\medskip

\noindent\textit{Case $i_{1}=0$.} Note that, if we 
omit the first factor in this case, then $U_{i_{2}}(\xmpar) 
\ematrix_{\bvec}\dots\ematrix_{\bvec} U_{i_{k}}(\xmpar)$ 
is a summand of $W_{t{-}2}(\xmpar)$ for the sequence of 
indices in $C(t{-}2,k{-}1)$. Thus,
\[
U_{i_{1}}(\xmpar)\ematrix_{\bvec}\dots 
\ematrix_{\bvec}U_{i_{k}}(\xmpar)=\ematrix_{\bvec} 
\cdot\left(U_{i_{2}}(\xmpar)\ematrix_{\bvec}\dots 
\ematrix_{\bvec}U_{i_{k}}(\xmpar)\right),
\]
since $U_{0}(\xmpar)=\imatrix$.
\medskip

\noindent The sum of the $W_{t}(\xmpar)$ terms 
appearing in the first two cases equals the part 
$\xmpar W_{t{-}1}(\xmpar)-W_{t{-}2}(\xmpar)$ in the recursion, 
while the sum of the terms in the third case equals the 
part $\ematrix_{\bvec}W_{t{-}2}(\xmpar)$. 
Thus, we obtain that 
$W_{t}(\xmpar)$ satisfies \eqref{eq:matrixcp-recursion} 
and is equal to $U_{t}^{\bvec}$.
\end{proof}

Outside of the case of type $\typeADE$ graphs, 
koszulity of $\bzigzag$ is 
obtained very similarly to the koszulity of $\zigzag$ 
with \fullref{proposition:koszul-abstract}, i.e.:

\begin{proposition}\label{proposition:koszul-boundary}
If $\Gg$ is not a type $\typeADE$ graph, then $\bzigzag$ is Koszul.
\end{proposition}

Note that, morally speaking, \eqref{eq:matrixcp-recursion} 
and \fullref{proposition:koszul-boundary} imply that 
the number of projective indecomposables in the minimal linear projective resolutions 
grows faster if one increases the number of 
vertices having a vertex 
conditions, and also 
the closer one gets to such vertices, 
cf. \fullref{example-koszul-finiteD}.

\begin{proof}
As argued in the proof of 
\fullref{proposition:koszul-abstract}, since 
$\Gg$ is not a type $\typeADE$ graph, we know 
that $U_{t}(\amatrix)$ will always be a non-negative integral 
matrix without zero columns or rows.
Thus, by \fullref{lemma:relation-polynomials}, 
the same is true for $U_{t}^{\bvec}(\amatrix)$, because all 
summands in \eqref{eq:equality-polynomials} evaluated at 
$\amatrix$ have non-negative entries with the leading 
term $U_{t}(\amatrix)$ also having non-zero columns and rows.
Hence, if we start with $\algstuff{K}_{0}^{\bvec}=\lsimple[\ii]$, then 
we know that $\underline{a}_{t}=U_{t}^{\bvec}(\amatrix)\underline{a}_0$ 
will never be zero, and we are done by the same reasoning as in 
the proof of \fullref{proposition:koszul-abstract}.
\end{proof}

\subsubsection{Construction -- Part II}\label{subsubsec:koszul-ADE-construction-b}

This time there are no cases that need to be eliminated 
(as already indicated in \fullref{example-koszul-finiteD}), 
since all $\Gg$ and all $\bvec$ will give 
Koszul algebras. Before we can show this, we recall that the 
Chebyshev polynomials have a closed form, namely
\begin{gather}\label{eq:cp-formula}
U_{t}(\xpar)={\textstyle\sum_{k=0}^{s(t)}}\,
(-1)^{k}{\textstyle\binom{t-k}{k}}\,\xpar^{t-2k},
\text{ for }t\in\Z_{\geq 0},
\end{gather}
where $s(t)=\lfloor t/2\rfloor$. To get a matrix version of 
\eqref{eq:cp-formula} we let $\xmpar$ and $\ympar$ denote 
two non-commuting variables of degrees $1$ and $2$, respectively.
Let $M(t,k)$ be the set of monomials in $\xmpar$ and $\ympar$ 
of degree $t$ with $k$ different $\ympar$-factors. (For example, $M(t,0)=\{\xmpar^{t}\}$ 
and $M(t,1)=\{\xmpar^{t{-}2}\ympar,\xmpar^{t{-}3}\ympar\xmpar,\dots,\ympar\xmpar^{t{-}2}\}$.)

\begin{lemma}\label{lemma:koszul-boundary-poly}
We have
\begin{gather}\label{eq:cpb-formula}
U_{t}^{\bvec}(\xmpar)={\textstyle\sum_{k=0}^{s(t)}}\,
(-1)^{k}{\textstyle\sum_{m\in M(t,k)}}\,m(\xmpar,\imatrix-\ematrix_{\bvec}),
\text{ for }t\in\Z_{\geq 0},
\end{gather}
where $\xmpar$ is to be assumed to be a matrix.
\end{lemma}

Indeed, \eqref{eq:cpb-formula}
specializes to \eqref{eq:cp-formula} in case $\ematrix_{\bvec}$ 
is the zero matrix.

\begin{proof}
One immediately checks that $U_{0}^{\bvec}(\xmpar)$ and 
$U_{1}^{\bvec}(\xmpar)$ satisfy \eqref{eq:cpb-formula}. Analyzing formula 
\eqref{eq:cpb-formula}, one sees that the monomials 
in $U_{t}^\bvec(\xmpar)$ can be obtained by multiplying 
the ones from $U_{t{-}1}^\bvec(\xmpar)$ with $\xmpar$ from 
the left and subtracting the ones from $U_{t{-}2}^\bvec(\xmpar)$ 
multiplied with $(\imatrix-\ematrix_{\bvec})$ from the 
left, which is exactly the recursion \eqref{eq:cpb-recursion-formula}. 
Thus, the claim follows.
\end{proof}

To obtain positivity of $U_{t}^{\bvec}(\amatrix)$ we 
need a combinatorial interpretation of the summands 
in \eqref{eq:cpb-formula} in terms of certain paths, where 
we recall that e.g. $\pathxx{i}{j}{i}$ denotes a path 
in the double of $\Gg$, which we from now on identify 
with paths in $\Gg$, by convention.

\begin{definition}\label{definition:koszul-inflow}
For each vertex $\ii\notin\bvec$ choose 
an edge of $\Gg$ that is incident with $\ii$. If every edge 
of $\Gg$ is chosen by at most one vertex, then we 
call this a \textit{singleton inflow} (outside of $\bvec$).

For a fixed choice of singleton 
inflow, we call a path $\pathxx{i}{j}{i}$ a \textit{chosen zigzag
at $\ii\notin\bvec$} if the arrow 
$\pathx{j}{i}$ was the choice for the singleton inflow at the vertex $\ii$.
\end{definition}

\begin{example}\label{example:koszul-inflow}
A singleton 
inflow is not unique and might not exist at all:
a type $\typeADE$ graph possesses at least one singleton inflow 
as long as $\bvec\neq\emptyset$, but if we would allow $\bvec=\emptyset$, then
the type $\typeA[]$ graph has none at all, while a type $\atypeA[]$ graph has exactly two.
\end{example}

\begin{lemma}\label{lemma:zigzag-nooverlap}
Assume that $\Gg$ has a singleton inflow and 
let $p$ be a path in $\Gg$. Then two distinct 
chosen zigzags in $p$ cannot have any edges of the path in common.
\end{lemma}

\begin{proof}
We can immediately reduce this statement to a 
path of length $3$, i.e. $p=\pathxxx{i}{j}{k}{l}$ 
and assume that $\pathxx{i}{j}{k}$ and $\pathxx{j}{k}{l}$ are 
chosen zigzags. This forces $\ii=\ii[k]$ and $\ii[j]=\ii[l]$, thus 
$p=\pathxxx{i}{j}{i}{j}$ with the second and third 
edge both being chosen for the singleton inflow. 
This is a contradiction, since they are both 
the edge connecting $\ii$ and $\ii[j]$.
\end{proof}

For a fixed choice of a singleton 
inflow, we 
associate to a path  
$p=\pathxxxx{i_{1}}{i_{2}}{\dots}{i_{t{-}1}}{i_{t}}$
a monomial $m(p)$ in non-commuting variables 
$\xmpar$ and $\ympar$ by first substituting any 
chosen zigzags in $p$ by $\ympar$, and 
afterwards all remaining edges by $\xmpar$.

\begin{example}\label{example:koszul-inflow-2}
Take the type $\typeA[4]$ graph with $\bvec=\{\ii[1]\}$. Then there 
is a unique singleton inflow, which we illustrate by 
orient an edge towards a vertex in case it is the chosen one 
for that vertex, i.e.
\[
\begin{tikzcd}[ampersand replacement=\&,row sep=large,column sep=tiny,arrows={shorten >=-.5ex,shorten <=-.5ex},labels={inner sep=.05ex}]
{\color{myblue}\ii[1]}
\arrow[thick,myred,->]{r}{}
\&
{\color{myred}\ii[2]}
\arrow[thick,mypurple,->]{r}{}
\&
{\color{mypurple}\ii[3]}
\arrow[thick,mygreen,->]{r}{}
\&
{\color{mygreen}\ii[4]}
\end{tikzcd}
,\text{ for }
\bvec=\{\ii[1]\}.
\]
With this singleton inflow, for example, the path $\pathxxx{3}{2}{1}{2}$ in $\Gg$
is associated to $\xmpar\ympar$, with $\ympar$ corresponding 
to $\pathxx{2}{1}{2}$. Similarly, in case $\bvec=\{\ii[1],\ii[2]\}$, then 
a singleton inflow would be
\[
\begin{tikzcd}[ampersand replacement=\&,row sep=large,column sep=tiny,arrows={shorten >=-.5ex,shorten <=-.5ex},labels={inner sep=.05ex}]
{\color{myblue}\ii[1]}
\arrow[thick,-]{r}{}
\&
{\color{myred}\ii[2]}
\arrow[thick,mypurple,->]{r}{}
\&
{\color{mypurple}\ii[3]}
\arrow[thick,mygreen,->]{r}{}
\&
{\color{mygreen}\ii[4]}
\end{tikzcd}
,\text{ for }
\bvec=\{\ii[1],\ii[2]\},
\]
and $\pathxxx{3}{2}{1}{2}$ would be associated to 
$\xmpar^{3}$.
\end{example}

The next lemma yields a combinatorial 
interpretation of $m(\amatrix,\imatrix-\ematrix_{\bvec})$ for $m\in M(t,k)$.

\begin{lemma}\label{lemma:m-interpretation}
Let $m\in M(t,k)$ for $t\in\Z_{\geq 0}$ and 
$0\leq k\leq s(t)$. Then the $\ii$-$\ii[j]$ position of 
$m(\amatrix,\imatrix-\ematrix_{\bvec})$ is equal 
to the number of paths $p=\pathxxxx{j}{i_{1}}{\dots}{i_{t{-}1}}{i}$ 
(of length $t$) from $j$ to $i$ such that $m(p)=m$.
\end{lemma}

\begin{proof}
Let $m=\xmpar^{a_{1}}\ympar^{b_{1}}\xmpar^{a_{2}}\ympar^{b_{2}}\dots
\xmpar^{a_{s}}\ympar^{b_{s}}$ such that the sum of all 
$a_{i}$ is $t-2k$ and the sum of all $b_{i}$ is $k$. By 
definition, the $\ii$-$\ii[j]$ position of $m(\amatrix,\imatrix-\ematrix_{\bvec})$ 
is equal to the number of paths of length $t-2k$ where the path 
made of the first $a_{1}+\dots+a_{i}$ edges ends in a vertex 
outside of $\bvec$, for all $\ii$. Such a path can be uniquely 
extended to length $t$, by adding $b_{1}$ chosen 
zigzags after the first $a_{1}$ edges, then $b_{2}$ chosen 
zigzags after the next $a_{2}$ edges, etc. The resulting 
path $p$ satisfies $m(p)=m$, by construction, and clearly 
any such path can be obtained uniquely in such a way, if $m$ is fixed.
\end{proof}

Note that a path from $\ii[j]$ to $\ii$ can contribute to 
multiple $m(\amatrix,\imatrix-\ematrix_{\bvec})$ for 
$m\in M(t,k)$ via \fullref{lemma:m-interpretation}, 
and of course all path contribute to $m(\amatrix)=\amatrix^{t}$.

\begin{proposition}\label{proposition:ADE-koszul-boundary}
If $\Gg$ is a type $\typeADE$ graph, then $\bzigzag$ is Koszul.
\end{proposition}

We stress that the proof will use the fact that $\bvec\neq\emptyset$.

\begin{proof}
We first fix a choice of a singleton inflow, 
which exists by \fullref{example:koszul-inflow}, and we use the description 
of $U_{t}^{\bvec}(\amatrix)$ from \eqref{eq:cpb-formula}. 
Via this and \fullref{lemma:m-interpretation} the 
$\ii$-$\ii[j]$ position of $U_{t}^{\bvec}(\amatrix)$ is an alternating 
sum of numbers of certain paths of length $t$ from 
$\ii[j]$ to $\ii$ starting with the total 
number of paths of length $t$ for $k=0$.

We start with the following claims.
\medskip

\noindent\textit{Claim 1.} $U_{t}^{\bvec}(\amatrix)$ 
has only non-negative entries.
\medskip

\noindent\textit{Proof of the claim.}
The strategy is to show that for each path of length 
$t$ the contributions for each $k$ in the sum now either 
give $0$ or $1$. For this purpose, fix a path $p$ of 
length $t$. Now let $k$ be such that $m(p)\in M(t,k)$. 

We first claim that $p$ does not give a contribution 
to any other $m\in M(t,l)$ for $l\geq k$. This is 
clear for $l>k$, since $k$ is exactly the number of 
chosen zigzags in $p$. For $l=k$ the chosen zigzags 
would need to be at different positions, which 
is also not possible. 

Next, we claim that for $l<k$ there are exactly 
$\binom{k}{l}$ different $m\in M(t,l)$ that count
$p$. These $m$ are obtained from $m(p)$ by replacing $k-l$ of the 
occurring $\ympar$ by $\xmpar^{2}$.

Summing all of this up, we see that for $k>0$
the contribution of $p$ to the $\ii$-$\ii[j]$ position of 
$U_{t}^{\bvec}(\amatrix)$ is 
$1-\binom{k}{1}+\binom{k}{2}-\ldots+(-1)^{k}\binom{k}{k}=0$, 
while for $k=0$ the contribution is $1$. 
Thus, in total, we obtain that $U_{t}^{\bvec}(\amatrix)$ 
only has non-negative entries, which proves Claim 1.
\medskip

\noindent\textit{Claim 2.} $U_{t}^{\bvec}(\amatrix)=0$ 
if and only if $\Gg$ is a type $\typeA[]$ graph and $\bvec\neq\{\ii[1]\}$, 
respectively $\bvec\neq\{\ii[n]\}$.
\medskip

\noindent\textit{Proof of the claim.}
Assume that $\Gg$ has a trivalent vertex. Then, for 
any vertex $\ii$, we can construct a path of arbitrary 
length that contains no chosen zigzag for any choice 
of a singleton inflow starting at that vertex. This is 
due to the fact that the trivalent vertex has two edges 
that are not chosen for it, which one can use to construct such 
paths.

Assume $\Gg$ is of type $\typeA[n]$ and $\bvec\neq\{\ii[1]\}$, 
respectively $\bvec\neq\{\ii[n]\}$. Then for any singleton 
inflow choice there exists at least one of the 
following cases:
Either there exists a vertex with a vertex 
condition which has two 
neighbors, again allowing to construct an arbitrary 
long path without chosen zigzags. Or both leafs have a 
vertex condition in which 
case the edge incident to one of them is not chosen,
and we can construct again an 
arbitrary long path not containing a chosen zigzags.

Finally, we consider the case of a type $\typeA[n]$ graph with $\bvec=\{\ii[1]\}$, 
the case $\bvec=\{\ii[n]\}$ follows by symmetry. In this case 
there is a unique singleton inflow. For this 
singleton inflow there can not be any path of length $\geq 2n-1$ 
not containing any chosen zigzag. Hence, $U_{t}^{\bvec}(\amatrix)=0$ 
for $t\geq 2n-1$, which shows the claim.

Altogether this implies that outside of graphs of type $\typeA[n]$ 
with $\bvec=\{\ii[1]\}$, respectively $\bvec=\{\ii[n]\}$, $\bzigzag$ is 
Koszul by the same argument as for \fullref{proposition:koszul-boundary}.

The hardest case remains: a type $\typeA[n]$ graph with 
$\bvec=\{\ii[1]\}$, because the case $\bvec=\{\ii[n]\}$ 
follows again by symmetry. Observe that there 
is a unique longest path not containing any chosen zigzag, which is 
of the form $\pathxxxx{i}{\dots}{n}{\dots}{1}$, and which is of length 
$(2n+1)-i$. Thus, the matrix $U_{(2n+1)-i}^{\bvec}(\amatrix)$ 
contains only one non-zero entry in the $i^{\mathrm{th}}$ column 
which is located in the 
$\ii[1]$-$\ii[i]$ position, 
which shows that this case is also Koszul.
\end{proof}

\makeautorefname{proposition}{Propositions}

\begin{remark}\label{remark:global-dimension}
From the proofs of \fullref{proposition:koszul-boundary} 
and \ref{proposition:ADE-koszul-boundary} we obtain 
the length of a minimal projective resolution for all 
simple modules, which are given by Chebyshev polynomials.
This shows that
$\bzigzag$ has infinite global dimension 
unless $\Gg$ is a type $\typeA[n]$ graph with 
$\bvec = \{\ii[1]\}$, respectively $\bvec=\{\ii[n]\}$, 
in which case it has global dimension $2n-1$. 
Note that this is known for the type $\typeA[n]$ graph 
with the vertex condition imposed on one leaf due to its connection 
to category $\mathcal{O}$, but in general this appears to be a new observation.
\end{remark}

\makeautorefname{proposition}{Proposition}

\begin{example}\label{example:koszul-finite}
If $\Gg$ is a type $\typeA[]$ graph with $\bvec=\{\ii[1]\}$ 
or $\bvec=\{\ii[n]\}$, then the linear projective resolutions are all finite. 
This follows from \fullref{proposition:ADE-koszul-boundary} 
and \fullref{theorem:qhb}.
To see this explicitly (we do the case $\bvec=\{\ii[1]\}$, the other follows by symmetry), 
we observe that the columns of 
the matrices (exemplified in case of the type $\typeA[4]$ graph)
\[
\xy
(0,0)*{
\scalebox{.8}{$\displaystyle
\left(
\begin{array}{cccc}
1 & 0 & 1 & 0
\\
0 & 1 & 0 & 1
\\
1 & 0 & 1 & 0
\\
0 & 1 & 0 & 0
\end{array}
\right)
$}};
(0,-10)*{\text{$t=2$}};
\endxy
,\;
\xy
(0,0)*{
\scalebox{.8}{$\displaystyle
\left(
\begin{array}{cccc}
0 & 1 & 0 & 1
\\
1 & 0 & 1 & 0
\\
0 & 1 & 0 & 0
\\
1 & 0 & 0 & 0
\end{array}
\right)
$}};
(0,-10)*{\text{$t=3$}};
\endxy
,\;
\xy
(0,0)*{
\scalebox{.8}{$\displaystyle
\left(
\begin{array}{cccc}
1 & 0 & 1 & 0
\\
0 & 1 & 0 & 0
\\
1 & 0 & 0 & 0
\\
0 & 0 & 0 & 0
\end{array}
\right)
$}};
(0,-10)*{\text{$t=4$}};
\endxy
,\;
\xy
(0,0)*{
\scalebox{.8}{$\displaystyle
\left(
\begin{array}{cccc}
0 & 1 & 0 & 0
\\
1 & 0 & 0 & 0
\\
0 & 0 & 0 & 0
\\
0 & 0 & 0 & 0
\end{array}
\right)
$}};
(0,-10)*{\text{$t=5$}};
\endxy
,\;
\xy
(0,0)*{
\scalebox{.8}{$\displaystyle
\left(
\begin{array}{cccc}
1 & 0 & 0 & 0
\\
0 & 0 & 0 & 0
\\
0 & 0 & 0 & 0
\\
0 & 0 & 0 & 0
\end{array}
\right)
$}};
(0,-10)*{\text{$t=6$}};
\endxy
,\;
\xy
(0,0)*{
\scalebox{.8}{$\displaystyle
\left(
\begin{array}{cccc}
0 & 0 & 0 & 0
\\
0 & 0 & 0 & 0
\\
0 & 0 & 0 & 0
\\
0 & 0 & 0 & 0
\end{array}
\right)
$}};
(0,-10)*{\text{$t\geq 7$}};
\endxy,
\]
together with the starting matrices $\imatrix$ for $t=0$ and $\amatrix$ for $t=1$,
give the summands in the corresponding linear projective resolutions. 
In fact, having these matrices it is easy to write down the resolutions 
using the same methods as in \fullref{example-koszul-finiteD}.
\end{example}

\subsubsection{The proof}\label{subsubsec:koszul-proof-boundary}

\makeautorefname{proposition}{Propositions}

\begin{proof}[Proof of \fullref{theorem:koszulb}]
We combine \fullref{proposition:koszul-boundary} and \ref{proposition:ADE-koszul-boundary}.
\end{proof}

\makeautorefname{proposition}{Proposition}
%

\end{document}